\documentclass[a4]{article}

\usepackage{amsmath}
\usepackage{amsthm}
\usepackage{amssymb}
\usepackage{tikz-cd}
\usepackage{mathtools}
\usepackage{stix2}
\usepackage{bbold}
\usepackage{tikz}
\usepackage[
style=alphabetic,
maxcitenames=10,
maxbibnames=10,
giveninits=true,
backend=biber
]{biblatex}
\usepackage[
colorlinks=true,
citecolor=blue,
linkcolor=red,
linktoc=page
]{hyperref}
\usepackage{cleveref}

\crefname{equation}{}{}
\Crefname{equation}{}{}
\crefname{lem}{lemma}{lemmas}
\crefname{prop}{proposition}{propositions}
\crefname{thm}{theorem}{theorems}
\crefname{cor}{corollary}{corollaries}
\crefname{definition}{definition}{definitions}

\pgfdeclarearrow{
	name = pxto,
	setup code = {
		\pgfarrowssettipend{1.5\pgflinewidth}
		\pgfarrowssetbackend{-2.5508\pgflinewidth}
		\pgfarrowssetlineend{-.25\pgflinewidth}
		\pgfarrowssetvisualbackend{-0.021\pgflinewidth}
		\pgfarrowsupperhullpoint{1.5\pgflinewidth}{0\pgflinewidth}
		\pgfarrowsupperhullpoint{-2.0085\pgflinewidth}{3.6525\pgflinewidth}
		\pgfarrowsupperhullpoint{-2.5508\pgflinewidth}{3.0763\pgflinewidth}
	},
	drawing code = {
		\pgfsetdash{}{0pt}%
		\pgfpathmoveto{\pgfpoint{1.5\pgflinewidth}{0.0254\pgflinewidth}}%
		\pgfpathlineto{\pgfpoint{-2.0085\pgflinewidth}{3.6525\pgflinewidth}}%
		\pgfpathlineto{\pgfpoint{-2.5508\pgflinewidth}{3.0763\pgflinewidth}}%
		\pgfpathlineto{\pgfpoint{-0.4322\pgflinewidth}{0.5\pgflinewidth}}%
		\pgfpathlineto{\pgfpoint{-0.4322\pgflinewidth}{-0.5\pgflinewidth}}%
		\pgfpathlineto{\pgfpoint{-2.5508\pgflinewidth}{-3.0763\pgflinewidth}}%
		\pgfpathlineto{\pgfpoint{-2.0085\pgflinewidth}{-3.6525\pgflinewidth}}%
		\pgfpathclose%
		\pgfusepathqfill
	}
}
\tikzset{>=pxto}
\tikzcdset{arrow style=tikz}

\newcommand{\nocontentsline}[3]{}
\let\origcontentsline\addcontentsline
\newcommand\stoptoc{\let\addcontentsline\nocontentsline}
\newcommand\resumetoc{\let\addcontentsline\origcontentsline}

\counterwithin{equation}{section}
\newtheorem{prop}[equation]{Proposition}
\newtheorem{theorem}[equation]{Theorem}
\newtheorem{lem}[equation]{Lemma}
\newtheorem{cor}[equation]{Corollary}
\theoremstyle{definition}
\newtheorem{definition}[equation]{Definition}

\newtheorem{rem}[equation]{Remark}

\DeclareMathOperator{\Hom}{Hom}
\DeclareMathOperator{\Map}{Map}

\DeclareMathOperator{\Fun}{Fun}

\DeclareMathOperator*{\colim}{colim}

\DeclareMathOperator{\im}{im}

\DeclareMathOperator{\Cat}{Cat}
\DeclareMathOperator{\op}{op}

\DeclareMathOperator{\Ar}{Ar}

\DeclareMathOperator{\Flow}{Flow}

\DeclareMathOperator*{\amalga}{\amalg}
\DeclareMathOperator{\timest}{\times}

\newcommand{\cA}{\mathcal{A}}
\newcommand{\cC}{\mathcal{C}}
\newcommand{\cD}{\mathcal{D}}
\newcommand{\cS}{\mathcal{S}}
\newcommand{\cL}{\mathcal{L}}
\newcommand{\cR}{\mathcal{R}}
\newcommand{\cF}{\mathcal{F}}
\newcommand{\cI}{\mathcal{I}}
\newcommand{\cT}{\mathcal{T}}
\newcommand{\cQ}{\mathcal{Q}}

\newcommand{\IK}{\mathcal{IK}}
\newcommand{\CSS}{\mathcal{CSS}}
\newcommand{\CsSS}{\mathcal{CSS}_+}

\newcommand{\del}{\partial}
\newcommand{\join}{\star}
\newcommand{\altjoin}{\diamond}
\newcommand{\Simp}{\mathbb{\Delta}}
\newcommand{\Horn}{\Lambda}
\newcommand{\Set}{Set}

\title{Quasi-unital inner Kan spaces}
\author{Trygve Poppe Oldervoll}
\date{}

\begin{document}
	\maketitle
	\begin{abstract}
		We show that semi-simplicial spaces that i) admit inner horn fillers up to homotopy and ii) possess units in a weak sense provide a viable model for $\infty$-categories. The existence of units can be expressed through various quasi-unitality conditions, and we compare the natural generalization of three such conditions found in the literature. This work is motivated by applications in Floer homotopy theory.
	\end{abstract}
	
	\tableofcontents
	
	\section{Introduction}\label{sec:introduction}
	In this paper, we study models for $\infty$-categories where units are encoded as a property rather than structure. To illustrate what we mean by this, consider two well-known models for $\infty$-categories: quasi-categories and Segal spaces. Both of these models are based on simplicial objects (sets or spaces) where an $n$-cell corresponds to a string of $n$ composable morphisms. The compositional structure is encoded by face maps, while the structure of units is encoded by the degeneracy maps. In particular, the degeneracy map $s_0 \colon X_0 \to X_1$ encodes a choice of unit for each object. Moreover, a functor (simplical map) $F\colon X\to Y$ is required to respect this choice. To weaken the rigidity of units, we consider \emph{semi-simplicial} objects. Semi-simplicial objects satisfying analogues of the quasi-category condition or the Segal condition can be thought of as representing $\infty$-\emph{semi-categories}, i.e. higher categories with composition but not necessarily units. To recover a theory of $\infty$-categories, we impose a condition of \emph{quasi-unitality}, which guarantees the existence of units. It turns out that the behaviour of units with respect to composition guarantees that if units exist, they are unique. Quasi-unitality has been studied in the context of quasicategories \cite{henry2018weak} \cite{steimle2018degeneracies} \cite{tanaka2018functors}, as well as in the context of (higher) Segal spaces \cite{harpaz2015quasi} \cite{haugseng2021segal}. Also in the quasi-category model, \cite{AB} makes use of a sufficient condition for quasi-unitality. The goal of this paper is to extend the theory of quasi-unitality for quasi-categories to the context of semi-simplicial spaces admitting inner horn fillers up to homotopy. We will show that this yields a workable model for the theory of $\infty$-categories, and we will show how the different conditions for quasi-unitality in this context relate. 
	
	One caveat is that we choose to work in the $\infty$-category of semi-simplicial spaces. This means that our theory relies on an underlying model for $\infty$-categories. This is no big loss as the foundations of $\infty$-categories are already well-developed in other models such as the quasi-category framework of \cite{HTT}. The arguments and results of this paper are presented in a model-independent way. We see our results as aligning with \cite{mazelgee2025model}, where the author speculates about a Joyal model structure on the $\infty$-category of simplical spaces. The main construction of this paper should be interpreted as a weak model structure (in the sense of \cite{henry2018weak}) on the $\infty$-category of \emph{marked} semi-simplicial spaces, with the main theorem showing that this indeed presents $\Cat_{\infty}$. Because the theory of weak model $\infty$-categories is yet to be developed, we will not state our results in these terms. However, many of our results and arguments are inspired by (weak) model theory. We also expect that the methods in this paper can be used to confirm this speculation made in \cite{mazelgee2025model}.
	
	Our motivation for studying quasi-unital semi-simplicial spaces stems from Floer homotopy theory, where such objects show up when constructing $\infty$-categories of \emph{flow categories}. A flow category should be thought of as a freely generated coherent chain complex in the ``ring'' of manifolds up to cobordism. This is incarnated as a category enriched in a certain monoidal category of manifolds with corners. In \cite{AB}, the authors construct a semi-simplicial set $\Flow$ whose 0-simplices are flow categories. Constructing the face maps here is relatively easy because they only involve passing to full subcategories and/or restricting to certain parts of the boundary of manifolds. Producing degeneracies for $\Flow$ relies on a more involved geometric construction, meaning that it is hard to produce the full coherence for a simplicial set. It is then easier to show that $\Flow$ is quasi-unital. 
	
	We now describe the three different conditions for quasi-unitality and state the main comparison results. We write $ss\cS=\Fun(\Simp_{s}^{\op}, \cS)$ for the $\infty$-category of semi-simplicial spaces. Generalizing \cite{steimle2018degeneracies}, we say that a semi-simplicial space with inner horn fillers is quasi-unital if every vertex admits an idempotent equivalence (see \cref{def:quasi_unital}). Generalising \cite{tanaka2018functors}, we say that a map $F\colon X\to Y$ is quasi-unital if it preserves idempotent equivalences. We then write $\IK_{qu} \subset ss\cS$ for the subcategory spanned by quasi-unital inner Kan spaces and quasi-unital maps between them.
	
	Let $ss\cS_{+}$ denote the $\infty$-category of marked semi-simplicial spaces (see \cref{def:marking}). We say that a marked semi-simplicial space $X$ is a \emph{marked inner Kan space} (see \cref{def:S_CSS+}) if
	\begin{enumerate}
		\item $X$ admits inner horn fillers.
		\item The marked edges are equivalences in $X$ (see \cref{def:equivalence_in_inner_Kan}).
		\item The collection of marked edges satisfies the 2-out-of-6 property.
	\end{enumerate}
	We write $\IK_{+} \subset ss\cS_{+}$ for the full subcategory spanned by such objects. Marked semi-simplicial sets satisfying these conditions are precisely the quasi-unital objects considered in \cite{henry2018weak}. In particular, they are the fibrant objects in the weak Joyal--Lurie model structure on $ss\Set_+$. This notion of quasi-unitality is closely related to the one studied in \cite{harpaz2015quasi}. 
	
	Given a semi-simplicial space $X$, we can obtain a marking on $X$ by marking all the equivalences in $X$. This defines a marked semi-simplicial space $X^{\natural}$. Note that this construction does not determine a functor because semi-simplicial maps are not required to preserve equivalences.
	
	\begin{theorem}\label{thm:quasi_unitality_conditions}
		A semi-simplicial space $X$ is a quasi-unital inner Kan space if and only if $X^{\natural}$ is a marked inner Kan space. A map $F\colon X\to Y$ between quasi-unital inner Kan spaces is quasi-unital if and only if it lifts to a map $F^{\natural}\colon X^{\natural} \to Y^{\natural}$ between marked semi-simplicial spaces.
	\end{theorem}
	\Cref{thm:outer_degenereacies} can be equivalently restated as saying that we have inverse equivalences
	\begin{align*}
		U : \IK_{+} \leftrightarrows \IK_{qu}: (-)^{\natural}
	\end{align*}
	where $U$ is the restriction of the forgetful functor $ss\cS_{+} \to ss\cS$.
	
	We also consider a sufficient condition for quasi-unitality. In \cite{AB} it is shown that if a semi-simplicial set $X$ admits horn fillers and there exist maps $s_{0},s_{n} \colon X_{n} \to X_{n+1}$ satisfying the expected simplicial identities, then $X$ lifts to a simplicial set. We generalize this condition to semi-simplicial spaces, saying that a semi-simplicial space $X$ \emph{admits outer degeneracies} if there exist maps $s_0, s_n \colon X_{n} \to X_{n+1}$ satisfying the simplicial identities up to suitably coherent homotopies (see \cref{def:outer_degeneracies}).
	\begin{theorem}\label{thm:outer_degenereacies}
		If an inner Kan space $X$ admits initial and terminal degeneracies, then for any vertex $x$, the edge $s_0 x$ is an idempotent equivalence at $x$. In particular, $X$ is quasi-unital. 
	\end{theorem}

	For semi-simplicial sets, one can use the extension results of \cite{steimle2018degeneracies} and \cite{AB} to implicitly do $\inf$-category theory with quasi-unital quasicategories. It is unclear if such extension results generalize to semi-simplicial spaces, i.e. if every object of $\IK_{qu}$ is in the image of the restriction functor $s\cS \to ss\cS$. Instead of pursuing such a result, we show that $\IK_{qu}\simeq \IK_{+}$ itself is a reasonable place to do $\infty$-category theory. We accomplish this by comparing $\IK_{+}$ to the full subcategory $\CsSS \subset ss\cS_{+}$ of \emph{complete semi-Segal spaces} as studied in \cite{harpaz2015quasi}. The inclusion of $\CsSS$ admits a left adjoint $\cL_{\CsSS} \colon ss\cS_{+} \to \CsSS$, and the main theorem of \cite{harpaz2015quasi} gives an equivalence $\CsSS \simeq \Cat_{\infty}$. 
	\begin{theorem}\label{thm:IK+_model}
		The composite functor
		\begin{align*}
			\IK_{qu} \simeq \IK_{+} \to \CsSS \simeq \Cat_{\infty}
		\end{align*}
		exhibits a localization at the trivial fibrations. Moreover, it preserves:
		\begin{enumerate}
			\item pullbacks where one leg is a fibration.
			\item inverse limits of Reedy fibrant diagrams.
			\item exponential objects $X^{J}$ where $J$ is a semi-simplicial set.
			\item slice objects $X_{p/}$ and $X_{/p}$ where $p\colon J \to X$ is any map from a semi-simplicial set.
		\end{enumerate}
	\end{theorem}
	\stoptoc
	\subsection*{Organization}
	The paper is organized as follows:
	\begin{itemize}
		\item In \cref{sec:simplicial_sets_and_spaces} we set up notation and basic constructions for (marked) (semi-)simplicial spaces.
		\item In \cref{sec:semi_kan_spaces} we show some basic results about the geometric realization of semi-simplicial spaces.
		\item In \cref{sec:marked_inner_kan_spaces}, we set up the ``weak model structure'' on marked semi-simplicial spaces and use this to prove basic properties of the localization $\cT$.
		\item In \cref{sec:slices_and_functor_categories} we study the behaviour of $\cT$ with respect to exponentials and slices.
		\item In \cref{sec:alternative_condition_for_quasi_unitality} we show how the different notions of quasi-unitality relate, proving \cref{thm:IK+_model} and \cref{thm:outer_degenereacies}.
	\end{itemize}
	\subsection*{Acknowledgements}
	The idea for this paper originated while I was pursuing joint work with Alice Hedenlund, and I would like to thank her for useful discussions. I would also like to thank Rune Haugseng for helpful conversations about $\infty$-categories, and for reading a draft of this paper. 
	\resumetoc

	\section{Simplicial sets and spaces}\label{sec:simplicial_sets_and_spaces}
	Throughout this paper, we fix a strongly inaccessible cardinal $\kappa$. We write $\Set$ and  $\Cat$ for the large categories of $\kappa$-small sets and categories, respectively. We take a model-independent view on $\infty$-categories, assuming the existence of (large) $\infty$-categories $\cS \subset \Cat_{\infty}$ of $\kappa$-small spaces and $\infty$-categories, respectively. For concreteness, one could fix the quasicategory model based on $\kappa$-small simplicial sets as in \cite{HTT}. We will freely use standard results about $\infty$-categories, such as the Yoneda lemma and the adjoint functor theorem. One key assumption we make about our underlying model is that $\Cat_{\infty}$ can be realized as the full subcategory of simplicial objects in $\cS$ spanned by \emph{complete Segal spaces}. This is justified by \cite{joyal2007quasi}. Throughout, we will treat $\Set \subset \cS$ and $\Cat \subset \Cat_{\infty}$ as full sub-categories spanned by the discrete objects.
	
	The \emph{simplex category} $\Simp$ has objects $[n] = \left\{ 0 \leq \dots \leq n \right\}$ for $0\leq n$, and morphisms order-preserving maps. We let $\Simp_{s}\subset \Simp$ denote the wide subcategory of injective maps. We denote the categories of (semi-)simplicial spaces and sets by
	\begin{align*}
		s\cS &= \Fun(\Simp^{\op}, \cS) & ss\cS &= \Fun(\Simp_{s}^{\op}, \cS) \\
		s\Set &= \Fun (\Simp^{\op}, \Set) & ss\Set &= \Fun(\Simp^{\op}_{s}, \Set )\, .
	\end{align*}
	Under the inclusion $\Set \subset \cS$, we think of a (semi-)simplicial set as also being a levelwise discrete (semi-)simplicial space.

	Let $X$ be a (semi-)simplicial space. We write $X_n = X([n])$ for its value at the $n$-simplex.
	Through the Yoneda embedding, we can think of $X$ as a functor $(s)s\cS^{\op} \to \cS$, at a (semi-)simplicial space, this takes the value
	\begin{align*}
		X(K) \coloneq \Map_{s\cS}(K, X) = \lim_{\sigma\colon \Delta^{n}\to K} X_n \,.
	\end{align*}
	Restriction along the inclusion $\Simp^{\op}_{s} \subset \Simp^{\op}$ induces a  forgetful functor
	\begin{align*}
		\cF\colon s\cS \to ss\cS \,.
	\end{align*}
	This has both left and right adjoints given by Kan extension, which we denote by $\cL$ and $\cR$, respectively.
	The $n$-simplex $\Delta^{n} \in s\Set$ is the image of $[n]$ under the Yoneda embedding $\Simp \to s\cS$
	Likewise, the semi-simplicial $n$-simplex $\Delta^{n}_s \in ss\Set$ is the image of $[n]$ under the Yoneda embedding $\Simp_{s} \to ss\cS$. It follows immedeatly that $\Delta^{n} = \cL(\Delta^{n}_s)$.
	Several basic simplicial sets, such as the $(i,n)$-horn $\Horn^{n}_{i}$ and the boundary $\del \Delta^{n}$, can be constructed by gluing simplices along injective maps. This means that the colimit diagrams defining them lift to $ss\cS$, so we have semi-simplicial counterparts $\del \Delta^{n}_{s}$ and $\Horn^{n}_{i,s}$ satisfying $\cL( \del \Delta^{n}_{s} ) = \del \Delta^{n}$ and $\cL(\Horn_{i,s}^{n}) = \Horn^{n}_{i}$. 
	We will use $T$ to denote the terminal semi-simplicial set $T_{n}=\ast$. This is also the underlying semi-simplicial set $\cF(\Delta^{0})$.
	\begin{rem}\label{rem:terminal_is_not_unit}
		Note that the terminal object $T$ is \textit{not} equivalent to the zero simplex $\Delta^{0}_{s}$. The zero simplex $\Delta^{0}_{s}$ is concentrated in degree $0$, meaning that there can be no map $T\to \Delta^{0}_{s}$. This is part of a phenomenon that makes semi-simplicial objects harder to work with than their simplicial counterparts: for degree reasons, there are not enough maps.
	\end{rem}
	For a simplicial space $X$, we can form its geometric realization, which we denote
	\begin{align*}
		\vert X \vert = \colim_{\Simp^{\op}}(X) \,.
	\end{align*}
	For a semi-simplicial space $Y$, we can form the realization $\Vert Y \Vert$ by instead taking the colimit over $\Simp^{\op}_{s}$. By \cite{HTT} Lemma 6.5.37, the inclusion $\Simp^{\op}_{s} \to \Simp^{\op}$ is cofinal, so restriction induces an equivalence
	\begin{align*}
		\Vert \cF X \Vert \overset{\simeq}\to \vert X \vert \,.
	\end{align*}
	The category $s\cS$ has a closed monoidal structure given by the cartesian product
	\begin{align*}
		(X\timest Y)_n = X_n \timest Y_n && (X^{Y})_n = \Map(\Delta^{n}\timest Y, X) \,.
	\end{align*}
	In the category $ss\cS$, the cartesian product is not the homotopically correct monoidal structure, but rather the geometric product $\otimes$ as in \cite{rourke1971delta}. This is uniquely defined by preserving colimits in each variable, and $\Delta^{n}_{s} \otimes \Delta^{m}_{s}$ being the nondegenerate simplices in $\Delta^{n}\timest \Delta^{m}$. This product has the property
	\begin{align*}
		\cL(X\otimes Y) = \cL(X) \timest \cL(Y),
	\end{align*}
	which holds for semi-simplicial sets by \cite{rourke1971delta}, and therefore for all $ss\cS$ by colimit preservation. We also get internal mapping objects 
	\begin{align*}
		(X^{Y})_n = \Map(\Delta^{n} \otimes Y,X)\, .
	\end{align*}
	
	Recall that a map of spaces $f\colon X\to Y$ is a \emph{monomorphism} if it is injective on $\pi_{0}$, and an equivalence on all higher homotopy groups for any choice of basepoint in $X$. This means that up to equivalence, $f$ is the inclusion of some connected components of $Y$. Dually, we say that $f$ is an \emph{epimorphism} if it is surjective on $\pi_{0}$. Any map of spaces $f\colon X\to Y$ has a unique factorization $X\to \im(f) \to Y$ as an epimorphism followed by a monomorphism.
	\begin{definition}\label{def:marking}
		A \emph{marking} of a semi-simplicial space $X$ is a monomorphism of spaces $A\to X_1$.
		A \emph{marked semi-simplicial space} is a pair $(X,A)$ where $A$ is a marking of $X$.
		A map of marked semi-simplicial spaces $(X,A) \to (Y,B)$ is a map of simplicial spaces $X\to Y$ such that the induced map $A\to X_1 \to Y_1$ factors through $B$ up to homotopy.
	\end{definition}
	
	\begin{rem}\label{rem:marked_fibrant}
		Instead of dealing with the notion of \textit{marked fibrant} semi-simplicial spaces as in \cite{harpaz2015quasi}, we include $A\to X_1$ being a monomorphism as part of the definition.
		This just means that whenever we perform any construction on marked simplicial sets that would produce a non-monic $f\colon A\to X_1$, we just replace it by its unique image $\im(f)\to X_1$.
	\end{rem}
	We denote the category of marked semi-simplicial spaces by $ss\cS_{+}$.
	This has an obvious forgetful functor
	\begin{align*}
		U\colon ss\cS_{+} \to ss\cS
	\end{align*}
	which has both a left adjoint $X^{\flat} = (X,\emptyset)$, and a right adjoint $X^{\sharp} = (X,X_1)$.
	The functor $(-)^{\sharp}$ has a further right adjoint
	\begin{align*}
		\widetilde{U} \colon ss\cS_{+} \to ss\cS
	\end{align*}
	which takes a marked semi-simplicial $(X,A)$ to the space of fully marked simplices in $X$, i.e.
	\begin{align*}
		\widetilde{U}(X)_n = \Map(\Delta^{n,\sharp}_{s}, X)\,.
	\end{align*}
	Restriction along $\Delta^{n,\flat}_{s} \to \Delta^{n,\sharp}_{s}$ gives a map $\widetilde{U}(X) \to U(X)$, and we can recover $X$ as the pushout
	\begin{align}\label{eq:marking_pushout}
		X \simeq U(X)^{\flat} \underset{\widetilde{U}(X)^{\flat}}{\amalga} \widetilde{U}(X)^{\sharp}\, .
	\end{align}
	
	By \cite{harpaz2015quasi} there are natural ways to mark $X\otimes Y$ and $X^{Y}$, giving a closed monoidal structure $\otimes$ on $ss\cS_{+}$ such that $U$ lifts to a monoidal functor.
	
	\begin{definition}\label{def:augmented}
		Let $\Simp_{a}$ denote the category whose objects are $[n]$ for $n\geq -1$ with the convention $[-1] =\emptyset$, and order-preserving maps. The category of \emph{augmented simplicial spaces} is the presheaf category $s\cS_{a}= \Fun(\Simp_{a}, \cS)$. Consider also the wide subcategory $\Simp_{s, a} \subset \Simp_{a}$ of injective morphisms, and the corresponding presheaf category $ss\cS_{a} = \Fun(\Simp_{a},\cS)$ of \emph{augmented semi-simplicial spaces}.
	\end{definition}
	Right Kan extension along the inclusion $\Simp_{(s)} \to \Simp_{(s),a}$ determines fully faithful functors 
	\begin{align*}
		s\cS \to s\cS_{a}, \quad ss\cS \to ss\cS_{a}
	\end{align*}
	with image spanned by the objects $X$ such that $X([-1]) \simeq \ast$. 
	
	The augmented simplicial category $\Simp_{a}$ carries a monoidal structure given by the \emph{join} of ordered sets, $[n]\join [m] = [n+m+1]$, whose unit is $[-1]$. Because the join of two injective maps is again injective, this restricts to a monoidal structure on $\Simp_{s,a}$. By Day convolution, this induces a monoidal structure on $s\cS_{a}$ which is uniquely determined by preserving colimits in each variable separately, and $\Delta^{n}\join\Delta^{m} = \Delta^{n+m+1}$ for $-1\leq n,m$. The Day convolution can be computed explicitly as
	\begin{align}\label{eq:join_formula}
		(J\join K)_n \simeq \coprod_{i+j = n-1} J_i \timest K_{j}\, .
	\end{align}
	We note in particular that if $J_{-1} \simeq K_{-1} \simeq \ast$, then $(J\join K)_{-1} \simeq \ast$, so $\join$ restricts to a monoidal structure on $s\cS$. 
	
	Similarly, we get a Day convolution monoidal structure on $ss\cS_{a}$, which restricts to a monoidal structure on $ss\cS$, also given by the formula \cref{eq:join_formula}. Because the inclusion $\Simp_{s,a} \to \Simp_{a}$ is monoidal, the left Kan extension along this functor gets a natural monoidal structure, which restricts to a monoidal structure on $\cL \colon ss\cS \to s\cS$. In particular, we have natural equivalences
	\begin{align*}
		\cL(J\join K ) &\simeq \cL(J)\join \cL(K)\, . \\
	\end{align*}
	If $(J,A)$ and $(K,B)$ are marked (semi-)simplicial spaces, we get a marking on $J\join K$ by the map
	\begin{align*}
		A\amalga B \to  J_1 \amalga K_1 \amalga J_0\timest K_0 = (J\join K)_1\, .
	\end{align*}
	
	As in \cite{HTT}, we also consider the \emph{alternative join} for semi-simplicial spaces given by the formula
	\begin{align*}
		J \altjoin K &= J \underset{(J\timest K \timest \Delta^{ \left\{ 0\right\} } )}{\amalga} (J \timest K \timest \Delta^{1}) \underset{(J\timest K \timest \Delta^{ \left\{ 1\right\} } )}{\amalga} K \, .
	\end{align*}
	The two notions of join give rise to two notions of \textit{slice}. The functor
	\begin{align}\label{eq:join_functor}
		J\join - \colon s\cS \to s\cS_{J/}
	\end{align}
	preserves colimits, and therefore admits a right adjoint. We denote the value of this right adjoint at a map $p\colon J\to X$ by $X_{p/}$, called the \emph{slice under} $p$. Explicitly $X_{p/}$ is given by the formula
	\begin{align*}
		(X_{p/})_{n} &= \Map(J\join\Delta^{n}, X ) \timest_{\Map(J,X)  } \left\{ p\right\}\, .
	\end{align*}
	By replacing \cref{eq:join_functor} with the join for (marked) semi-simplicial spaces, we get a notion of slice under a map $p\colon J\to X$ of (marked) semi-simplicial spaces.
	
	The alternative slice is similarly defined in terms of the right adjoint of
	\begin{align*}
		J\altjoin - \colon s\cS \to s\cS_{J/}\, ,
	\end{align*}
	and we write $X^{p/}$ for the vaule at a map $p\colon J\to X$. 
	
	There is a good reason for studying these two different notions of slice. The slice $X_{/p}$ is more natural from the simplicial point of view because $\Delta_{s}^{n-1}\join \Delta_{s}^{m-1} = \Delta_{s}^{n+m-1}$, so for instance when $J= \Delta^{0}$,
	\begin{align*}
		(X_{/p})_{n} = X_{n+1} \timest_{X_{0}} \left\{ p\right\}\, .
	\end{align*}
	To explicitly describe the simplices of $X^{p/}$ requires describing the simplices of $X^{\Delta^{1}}$, which involves decomposing the cylinder $\Delta^{n}\timest \Delta^{1}$. The slice $X_{p/}$ also has the big advantage of being defined for semi-simplicial spaces, while the $X^{p/}$ is harder to make sense of here because the maps used to define $J\altjoin K$ are not semi-simplicial. This could potentially be solved by introducing some factors $T^{\sharp}$, but this would make the slice $X^{p/}$ even more complicated to explicitly describe.

	The slice $X^{p/}$, on the other hand, is the correct slice from a categorical standpoint. The following lemma shows that $W^{p/}$ is the correct slice for internally reasoning about colimits in the complete Segal space $W$. (It agrees, for instance, with the slice used in \cite{rasekh2018introductioncompletesegalspaces} and \cite{MW2021}.)
	\begin{lem} \label{lem:alt_undercat_pullback}
		Let $W$ be a simplicial space, and $x\in W_{0}$ a vertex. Then the following square is a pullback.
		\begin{equation*}
			\begin{tikzcd}
				W^{x/} \arrow[r] \arrow[d] & W^{\Delta^{1}} \arrow[d, "ev_0"] \\
				\left\{ x\right\} \arrow[r] & W
			\end{tikzcd}
		\end{equation*}
		Let $J$ be a simplicial space, and $p\colon J \to W$ a map. We let $\widetilde{p}$ denote the point in $X^{J}$ determined by $p$. Then the following square is a pullback.
		\begin{equation*}
			\begin{tikzcd}
				W^{p/} \arrow[r] \arrow[d] & (W^{J} )^{\widetilde{p}/} \arrow[d] \\
				W \arrow[r]& W^{J}
			\end{tikzcd}
		\end{equation*}
	\end{lem}
	\begin{proof}
		Using the universal property of the join and exponentials, we have
		\begin{align*}
			W^{p/}(K) &\simeq \left\{ p\right\} \times_{W(J)} W(J\altjoin K) \\
			&\simeq \left\{ p\right\}\times_{W(J)} W(J) \times_{W(J\times K)} W(J\times K \times \Delta^{1}) \times_{W(J\times K)} W(K) \\
			&\simeq \left\{ p \right\} \times_{W^{J}(K)} W^{J\times \Delta^{1}}(K) \times_{W^{J}(K)} W(K).
		\end{align*}
		These equivalences are natural in $K$, so we have
		\begin{align*}
			W^{p/}= \Delta^{0}\timest_{W^{J}} W^{J\timest \Delta^{1} } \timest_{W^{J}} W.
		\end{align*}
		This gives the result for $J=\Delta^{0}$ and by a pasting argument, the result holds for arbitrary $J$.
	\end{proof}
	\begin{rem}\label{rem:mapping_space_from_slice}
		In particular, if $W$ is a complete Segal space, the fiber of $W^{x/}$ at a point $y\colon \Delta^{0} \to W$ is given by the fiber of $W^{\Delta^{1}} \to W\timest W$ at $(x,y)$, so its zeroth space is by definition the mapping space $\Map_{W}(x,y)$. 
	\end{rem}
	Finally, we consider \emph{free slices} of augmented semi-simplicial spaces. For $J\in ss\cS_{a}$, the functor
	\begin{align*}
		J\join - \colon ss\cS_{a} \to ss\cS_{a}
	\end{align*} 
	preserves colimits, and we write $X_{J/}$ for the value of its right adjoint at an object $X$. By definition, this has 
	\begin{align*}
		\Map(K,X_{J/}) \simeq \Map(J\join K, X)\, .
	\end{align*}
	As for the ordinary slice, $X_{J/}$ is covariantly functorial and limit preserving in the $X$ variable, and contravariantly functorial mapping colimits to limits in the $J$ variable.  
	Because $[-1]$ is initial in $\Simp_{a}$, evaluation at $[-1]$ is left adjoint to the functor mapping a space $S$ to the constant augmented semi-simplicial space $\underline{S}$ at $S$. For an object $X\in ss\cS_{a}$, and a point $x\in X_{-1}$, we define $X_{x}$ by the pullback
	\begin{equation*}
		\begin{tikzcd}
			X_{x} \arrow[r] \arrow[d]      &   X  \arrow[d] \\
			\underline{\ast}  \arrow[r, "\underline{x}"]               &   \underline{X_{-1}}
		\end{tikzcd}
	\end{equation*}
	where the right vertical is the unit of the adjunction. By evaluating this pullback at $[-1]$ we see that $X_{x}$ is a semi-simplicial space. This construction of $X_{x}$ is natural in pointed objects of $ss\cS_{a}$. Essentially by definition, we have the following lemma.
	\begin{lem}\label{lem:free_slice}
		For semi-simplicial spaces $J$ and $X$, we have $(X_{J/})_{-1} = \Map(J,X)$. With this identification, we have natural equivalences
		\begin{align*}
			(X_{J/})_{p} \simeq X_{p/}\, .
		\end{align*}
	\end{lem}
	There is a symmetric definition of the slice $X_{/J}$, which satisfies the obvious properties. 
	\section{Semi-Kan spaces}\label{sec:semi_kan_spaces}
	Before moving on to $\infty$-categories, we show how semi-simplicial spaces satisfying the Kan condition give a reasonable model for $\cS$. We begin by recalling the basics of lifting properties in $\infty$-categories.
	\begin{definition}
		Let $j\colon J\to K$ and $f\colon X\to Y$ be morphisms in an $\infty$-category $\cC$. We say that $j$ has the \emph{left lifting property} with respect to $f$, or equivalently that $f$ has the \emph{right lifting property} with respect to $j$ if for any solid diagram as below, there exists a dashed arrow making the diagram commute.
		\begin{equation}\label{eq:lifting_diagram}
			\begin{tikzcd}
				K \arrow[d, "j"] \arrow[r] & X\arrow[d, "f"]\\
				J \arrow[r] \arrow[ur, dotted] & Y
			\end{tikzcd}
		\end{equation}
		Equivalently, this means that the map of spaces
		\begin{align} \label{eq:lifting_pullback}
			\Map_{\cC}(J,X) \to \Map_{\cC}(K,X) \timest_{\Map_{\cC}(K,Y)} \Map_{\cC}(J,Y)
		\end{align}
		is an epimorphism.
		If the lift in \cref{eq:lifting_diagram} is unique up to homotopy, (i.e. if the map \cref{eq:lifting_pullback} is an equivalence) we say that $j$ is \emph{left orthogonal to} $f$.
	\end{definition}
	\begin{definition}
		Let $S$ be a collection of morphisms in an $\infty$-category $\cC$. We say that a morphism $f$ in $\cC$ is an $S$-\emph{fibration} if $f$ has the right lifting property with respect to every $j\colon K\to J$ in $S$. If $\cC$ has a terminal object $T$, we say that an object is $S$-\emph{fibrant} if the unique map $X\to T$ is an $S$-fibration.
	\end{definition}
	\begin{rem}\label{rem:closure_of_fibrations}
		The collection of $S$-fibrations is closed under (transfinite) composition, pullbacks, and retracts. The same holds for the collection of maps that are orthogonal to $S$, which is moreover closed under limits in $\Ar(\cC)$.
	\end{rem}
	\begin{definition}\label{def:S-cof}
		For a collection of morphisms $S$ in an $\infty$-category $\cC$, we say that a morphism $j\colon J\to K$ is an $S$-\emph{cofibration} if every $S$-fibration $f\colon X\to Y$ has the right lifting property with respect to $j$. We write $S$-cof for the class of $S$-cofibrations. If $\cC$ has an initial object $\emptyset$, we say that an object $J$ in $\cC$ is $S$-\emph{cofibrant} if the unique morphism $\emptyset \to X$ is an $S$-cofibration.
	\end{definition}
	\begin{rem}\label{rem:cof_weakly_saturated}
		The collection of $S$-cofibrations is closed under (transfinite) composition, pushouts, and retracts, making it a \emph{weakly saturated class}. By definition, every morphism in $S$ is an $S$-cofibration. For classes $S$ admitting a small object argument (see \cite{mazelgee2025model}), the class of $S$-cofibrations is in fact the smallest weakly saturated class containing $S$. This will be the case for all the lifting properties considered in this paper.
	\end{rem}
	
	In this section, we show how the usual lifting properties of the Kan--Quillen model structure on $s\Set$ extend to make $ss\cS$ a reasonable model for the theory of spaces. 
	
	\begin{definition}\label{def:kan_fibration}
		Consider the following collections of morphisms in $ss\cS$:
		\begin{enumerate}
			\item $J_{KQ} = \left\{\Horn^{n}_{s,i} \to \Delta^{n}, 0\leq i \leq n \right\}$.
			\item $I_{KQ} = \left\{\del \Delta_{s}^{n} \to \Delta_{s}^{n}, 0\leq n \right\}$.
		\end{enumerate}
		we call $J_{KQ}$-fibrations \emph{Kan fibrations}, $J_{KQ}$-cofibrations \emph{anodyne}, and $I_{KQ}$-fibrations trivial Kan fibrations.
	\end{definition}
	\begin{rem}
		Because the boundary inclusions generate all levelwise injections of semi-simplicial sets under pushouts and retracts, $I_{KQ}$-cof contains all levelwise injections between semi-simplicial sets. In particular, $J_{KQ} \subset I_{KQ}$-cof, so every trivial Kan fibration is in particular also a Kan fibration.
	\end{rem}
	\begin{rem}
		Because pushouts with one injective leg are preserved by the inclusion $\Set \to \cS$, any $I_{KQ}$-cofibration with levelwise discrete source must also have levelwise discrete target. In particular, the cofibrant objects (those such that $\emptyset \to J$ is a cofibration) are precisely the semi-simplicial sets. 
	\end{rem}
	\begin{definition}\label{def:semi_Kan}
		We say that a semi-simplicial space $X$ is a \emph{semi-Kan space} if the map $X\to T$ is a Kan fibration.
	\end{definition}

	We now recall some results about geometric realizations of semi-simplicial objects from \cite{SAG}.
	\begin{lem}[\cite{SAG} Lemma A.5.3.7]\label{lem:trivial_kan->geometric_eq}
		A trivial Kan fibration between semi-Kan spaces induces an equivalence on geometric realizations.
	\end{lem}
	\begin{lem}[\cite{SAG} Lemma A.5.4.3]\label{lem:simplicial_set_replacement}
		For any semi-simplicial space $X$, there exists a semi-simplicial set $\overline{X}$ and a trivial Kan fibration $\overline{X}\to X$.
	\end{lem}
	\begin{lem}\label{lem:semi_kan_realization}
		If in a pullback square of semi-simplicial spaces
		\begin{equation*}
			\begin{tikzcd}
				X \arrow[d] \arrow[r] & Y \arrow[d, "p"] \\
				Z \arrow[r] & W
			\end{tikzcd}
		\end{equation*}
		we have that $Y,Z$ and $W$ are semi-Kan spaces and that $p$ is a Kan fibration, then the square remains a pullback after geometric realization.
	\end{lem}
	\begin{proof}
		Apply \cite{SAG} Lemma A.5.4.1 to the hypercomplepte $\infty$-topos $\cS$.
	\end{proof}
	We say that a category $\cI$ is \emph{inverse} if there exists a functor $\cI^{\op} \to \cA$ where $\cA$ is a totally ordered ordinal, such that every non-identity morphism in $\cI$ is mapped to a non-identity morphism in $\cA$.
	\begin{definition}
		Let $\cI$ be an inverse category, and $\cC$ an $\infty$-category which admits limits, and $F\colon \cI \to \cC$ a diagram. For an $i\in \cI$, the \emph{matching morphism at } $i$ is the induced map
		\begin{align*}
			F(i) \to  M_{F}(i) \coloneq \lim\left ((\cI_{i/})^{\circ} \to \cI \xrightarrow{F} \cC \right),
		\end{align*} 
		where $(\cI_{i/})^{\circ} \subset \cI_{i/}$ is the full subcategory spanned by the non-initial objects.
	\end{definition}
	\begin{definition}
		Let $\cI$ be an inverse category, $\cC$ an $\infty$-category which admits limits, and $S$ a class of morphisms in $\cC$. We say that a natural transformation $\eta \colon F\to G$ between functors from $\cI$ to $\cC$ is an $S$-\emph{Reedy fibration} if for every $i\in \cI$, the map
		\begin{align*}
			F(i) \to M_{F}(i) \times_{M_{G}(i)} G(i)
		\end{align*}
		is an $S$-fibration. We say that a diagram $F\colon \cI \to \cC$ is $S$-\emph{Reedy fibrant} if the unique natural transformation to the constant functor $\underline{T}$ at the terminal object is an $S$-Reedy fibration.
	\end{definition}
	\begin{lem}\label{lem:inverse_limit_preserves_fib}
		Let $\cI$ be an inverse category, $\cC$ an $\infty$-category which admits limits, and $S$ a class of morphisms in $\cC$. If $F\to G$ is an $S$-Reedy fibration between $\cI$ diagrams in $\cC$, then the induced map $\lim(F) \to \lim(G)$ is an $S$-fibration in $\cC$. 
	\end{lem}
	\begin{proof}
		Let $J\to K$ be any map in $S$. By adjunction, the lifting problem
		\begin{equation*}
			\begin{tikzcd}
				J \arrow[r] \arrow[d]      &   \lim(F)  \arrow[d] \\
				K  \arrow[r] \arrow[ur, dashed]              &   \lim(K)
			\end{tikzcd}
		\end{equation*}
		in $\cC$ is equivalent to the lifting problem
		\begin{equation}\label{eq:inverse_limit_preserves_fib_1}
			\begin{tikzcd}
				\underline{J} \arrow[r] \arrow[d]      &   F  \arrow[d] \\
				\underline{K}  \arrow[r]               &   G
			\end{tikzcd}
		\end{equation}
		in $\Fun(\cI,\cC)$. Pick a map $\cI^{\op} \to \cA$ for a totally ordered ordinal $\cA$. This gives a total order on the objects of $\cI$ which we can induct over. For an object $i\in \cI$, write $\cI_{<i}\subset \cI$ for the full subcategory spanned by objects that are strictly smaller than $i$ in this order. Assume for induction that we have constructed the lift in \cref{eq:inverse_limit_preserves_fib_1} after restricting to the category $\cI_{<i}$. Note that because non-identity morphisms in $\cI$ increase order, 
		the functor $(\cI_{i/})^{\circ} \to \cI$ used to define the matching morphism factors through $\cI_{<i}$. The lift over $\cI_{<i}$ therefore produces, by adjunction, a lift
		\begin{equation*}
			\begin{tikzcd}
				J \arrow[r] \arrow[d]      &   M_{F}(i)  \arrow[d] \\
				K  \arrow[r] \arrow[ur, dashed]              &   M_{G}(i)\, .
			\end{tikzcd}
		\end{equation*}
		To extend the lift over $i$, it suffices by adjunction to solve the lifting problem
		\begin{equation*}
			\begin{tikzcd}
				J \arrow[r] \arrow[d]      &   F(i)  \arrow[d] \\
				K  \arrow[r]  \arrow[ur, dashed]             &   M_{F}(i)\times_{M_{G}(i)} G(i)
			\end{tikzcd}
		\end{equation*}
		in $\cC$, which we can do because $F\to G$ is an $S$-Reedy fibration. By induction, we can therefore extend the lift over all $\cI$.
	\end{proof}
	
	\begin{lem}\label{lem:inverse_lim_realization}
		Let $\cI$ be an inverse category, and $F\colon \cI \to ss\cS$ a $J_{KQ}$-Reedy fibrant diagram. Then the natural map
		\begin{align*}
			\Vert \lim(F) \Vert \to \lim \Vert F \Vert
		\end{align*}
		is an equivalence.
	\end{lem}
	\begin{proof}
		Using the notation $\cI_{<i}$ from the proof of \cref{lem:inverse_limit_preserves_fib}, note that the category $\cI_{<i}$ is itself an inverse category. Assume for induction that we have constructed an $I_{KQ}$-Reedy fibration
		\begin{align*}
			\overline{F}_{<i} \to F_{<i}
		\end{align*}
		whose target is the restriction of $F$ to $\cI_{<i}$, and whose source is valued in semi-simplicial sets. To extend this natural transformation over $i$, we use \cref{lem:simplicial_set_replacement} to pick a semi-simplicial set $\overline{F}(i)$ with a trivial Kan fibration
		\begin{align*}
			\overline{F}(i) \to F(i) \times_{M_{F}(i)} M_{\overline{F}(i)} \, .
		\end{align*}
		By induction, this produces an $I_{KQ}$-reedy fibration $\overline{F} \to F$, where $\overline{F}$ is valued in semi-simplicial sets. Each of the natural transformations $\overline{F} \to F \to \underline{T}$ are $J_{KQ}$-Reedy fibrations, so $\overline{F}$ is $J_{KQ}$-Reedy fibrant. Now assume for induction that we have constructed a functor 
		\begin{align*}
			F'_{<i} \colon \cI_{<i} \to s\Set
		\end{align*}
		and a natural equivalence $\cF \circ F'_{<i} \simeq \overline{F}_{<i}$. Then by Reedy fibrancy, the map
		\begin{align*}
			\overline{F}(i) \to M_{\overline{F}}(i) \simeq \cF(M_{F'}(i))
		\end{align*}
		is a Kan fibration, so by \cite{steimle2018degeneracies}, we may pick a simplicial structure on $\overline{F}(i)$ compatible with this map. By induction, we therefore get a diagram $F' \colon \cI \to s\Set$, and an equivalence $\cF\circ F' \simeq \overline{F}$. Because $\cI$ is an inverse category, the Reedy and injective model structures on $\Fun(\cI, s\Set)$ agree by \cite{bergner2013reedy}. In particular, this means that the strict limit of $F'$ is also a homotopy limit in the Kan--Quillen model structure, so in particular this limit is preserved by the localization $s\Set \to \pi_{0}\cS$ to the homotopy category of spaces. This localization factors through the $\infty$-category $\cS$ by the realization functor $\vert - \vert$, and since the homotopy category detects equivalences, the comparison map
		\begin{align*}
			\vert \lim F' \vert \xrightarrow{\simeq}  \lim \vert F' \vert 
		\end{align*}
		is an equivalence in $\cS$. Since the forgetful functor $\cF$ preserves limits and the homotopy type of realization, we similarly get
		\begin{align*}
			\Vert \lim \overline{F} \Vert \xrightarrow{\simeq} \lim \Vert \overline{F} \Vert \, .
		\end{align*}
		Now we apply \cref{lem:inverse_limit_preserves_fib} to the $J_{KQ}$-Reedy fibration $\overline{F} \to F$, and use \cref{lem:trivial_kan->geometric_eq} to conclude that the vertical maps in the following square are equivalences.
		\begin{equation*}
			\begin{tikzcd}
				\Vert \lim \overline{F} \Vert \arrow[r, "\simeq"] \arrow[d, "\simeq"]      &   \lim \Vert \overline{F} \Vert  \arrow[d, "\simeq"] \\
				\Vert \lim F \Vert \arrow[r]               &   \lim \Vert F \Vert
			\end{tikzcd}
		\end{equation*}
		By 2-out-of-3 for equivalences, the final map in this square must also be an equivalence.
	\end{proof}
	\begin{lem}\label{lem:unit_is_anodyne}
		If $X$ is a semi-Kan space and $J$ a semi-simplicial set, any $J\to X$ extends along the adjunction unit $u\colon J\to \cF\cL(J)$. 
	\end{lem}
	\begin{proof}
		By picking a trivial fibration $\overline{X} \to X$, it suffices to prove this for $X$ a Kan semi-simplicial set. By \cite{rourke1971delta}, $X$ lifts to a simplicial set $X'$. By adjunction, the lifting problem
		\begin{equation*}
			\begin{tikzcd}
				J \arrow[r] \arrow[d, "u"']      &   \cF(X') \\
				\cL\cF(J)  \arrow[ur,dashed]                
			\end{tikzcd}
		\end{equation*}
		is equivalent to 
		\begin{equation}\label{eq:unit_is_anodyne}
			\begin{tikzcd}
				\cL(J) \arrow[r] \arrow[d, "\cL(u)"']      &   X'.  \\
				\cL\cF\cL(J)  \arrow[ur,dashed]               & 
			\end{tikzcd}
		\end{equation}
		By the zig-zag identities, the counit at $\cL(J)$ gives a retraction of the vertical map, so the lifting problem \cref{eq:unit_is_anodyne} can always be solved.
	\end{proof}

	\section{Marked inner Kan spaces}\label{sec:marked_inner_kan_spaces}
	In this section, we set up our ``weak model structure'' on the category of marked semi-simplicial spaces and show that it presents $\Cat_{\infty}$. The definition relies on classes of cofibrations and anodynes, which we now define.
	\begin{definition}\label{def:generating_marked_cofibrations}
		The collection of \emph{generating marked cofibrations} is the collection $I_{JL}$ of maps in $ss\cS_{+}$ consisting of:
		\begin{enumerate}
			\item $\del \Delta_{s}^{n,\flat} \to \Delta_{s}^{n,\flat}, \, n \geq 0$.
			\item $\Delta^{1,\flat}_{s} \to \Delta^{1,\sharp}_{s}$.   
		\end{enumerate}
		We call $I_{JL}$-fibrations \emph{marked trivial fibrations}, and $I_{JL}$-cofibrations \emph{marked cofibrations}.
	\end{definition}
	\begin{lem}\label{lem:trivial_fib_cancellation}
		If $f$ and $g$ are composable morphisms in $ss\cS_{+}$ such that $f$ and $g\circ f$ are trivial fibrations, then so is $g$.
	\end{lem}
	\begin{proof}
		Let $J\to K$ be a marked cofibration between semi-simplicial sets, and $X\xrightarrow{f} Y \xrightarrow{g} X$ composable morphisms such that $f$ and $g\circ f$ are trivial fibrations. Then consider the following diagram, where the horizontal morphisms are arbitrary.
		\begin{equation*}
			\begin{tikzcd}
				\emptyset \arrow[r] \arrow[d]      &   X  \arrow[d, "f"] \\
				J   \arrow[d] \arrow[ur,dashed]            &   Y \arrow[d, "g"] \\
				K \arrow[r]   \arrow[uur,dotted]     & Z
				\arrow[from=2-1, to=2-2,crossing over]
			\end{tikzcd}
		\end{equation*}
		The dashed lift exists because $f$ is a trivial fibration, and $\emptyset \to J$ is a marked cofibration. The dotted lift exists because $J\to K$ is a marked cofibration and $g\circ f$ is a trivial fibration. Pushing the dotted lift forward along $f$ solves the desired lifting problem, showing that $g$ is a trivial fibration.
	\end{proof}
	
	\begin{rem}\label{rem:levelwise_injections}
		Because the maps in $I_{JL}$ are levelwise injective, every pushout, retract or composite of maps in $I_{JL}$ is levelwise injective. Hence, by \cref{rem:cof_weakly_saturated}, every marked cofibration is levelwise injective. Conversely, by using the pushout decomposition \cref{eq:marking_pushout} and by attaching cells, any levelwise injective map between marked semi-simplicial sets is a marked cofibration. As for $I_{KQ}$, the cofibrant objects are precisely the marked semi-simplicial sets.
	\end{rem}
	
	\begin{definition}\label{def:admissible_horn}
		A marked horn inclusion $(\Horn^{n}_{i,s}, A) \to (\Delta^{n}_{s}, B)$ is called \textit{admissible} if $A=B\cap (\Horn^{n}_{i,s})_1 $, and either
		\begin{enumerate}
			\item $A=\emptyset$ and $0< i < n$,
			\item $B$ consists of the edge $\left\{ 0,1\right\}$, and $i=0$,
			\item or $B$ consists of the edge $\left\{ n-1,n\right\}$, and $i=n$.
		\end{enumerate}
	\end{definition}
	\begin{definition}\label{def:S_CSS+}
		Let $J_{JL}$ denote the collection of morphisms of marked semi-simplicial sets that contains:
		\begin{enumerate}
			\item Admissible horn inclusions.
			\item $S_{2/6} \colon (\Delta^{3}_s, \left\{ \left\{ 0,2\right\}, \left\{ 1,3\right\}  \right\}) \to \Delta^{3,\sharp}_{s}$.
		\end{enumerate}
		We call $J_{JL}$-fibrations \emph{marked inner fibrations} and $J_{JL}$-cofibrations \emph{marked inner anodyne} morphisms.
	\end{definition}
	
	\begin{rem}\label{rem:2_out_of_6}
		We do not need to include the maps $S^{i}_{2/3}\colon (\Delta^{2},A)_{s}\to \Delta^{2,\sharp}_{s}, i=0,1,2$ that appear in \cite{harpaz2015quasi}, because the 2-out-of-6 property will turn out to be strictly stronger than the 2-out-of-3 property. In particular, the 2-out-of-6 property will guarantee that edges are marked if and only if they are equivalences (see \cref{prop:IK_+=IK_qu}).  
	\end{rem}

	\begin{definition}[\cite{harpaz2015quasi}]
		We say that a marked semi-simplicial space $W$ is a \textit{complete semi-Segal space} if $W\to T^{\sharp}$ is orthogonal to $J_{JL}$. We denote the full subcategory of $ss\cS_{+}$ spanned by the complete semi-Segal spaces by $\CsSS$.
	\end{definition}
	Let $W$ be a complete semi-Segal space. Orthogonality to the unmarked inner horn inclusions ensures that $W$ satisfies the Segal condition
	\begin{align*}
		W_{n} \simeq W_{1}\timest_{W_{0}} W_{1} \timest_{W_0} \dots \timest_{W_{0}} W_{1} \, .
	\end{align*}
	As shown in \cite{harpaz2015quasi}, orthogonality with respect to marked horns implies that all the marked edges in $W$ are invertible with respect to the semi-category structure on $W$ (here we mean invertible in the sene of \cite{harpaz2015quasi} 1.4.1, not to be confused with equivalences as in \cref{def:equivalence_in_inner_Kan}), while orthogonality with respect to $S_{2/6}$ implies that a all invertible edges are marked. By \cite{harpaz2015quasi}, all the inclusions $\Delta^{0}_{s} \to \Delta^{n,\sharp}_{s}$ are marked inner anodyne, and orthogonality with respect to theese implies that $W(\Delta^{\bullet,\sharp}_{s})$ is a constant semi-simplicial space at $W_{0}$. Because marked edges are precisely the invertible edges, this is analogous to the completeness condition for Segal spaces.

	Because the subcategory $\CsSS$ is defined as the local objects with respect to a class of morphisms, the inclusion $\CsSS \to ss\cS_{+}$ admits a left adjoint which we denote $\cL_{\CsSS}$.
	\begin{definition}\label{def:CSS+_equivalence}
		We say that a map of marked semi-simplicial spaces is a $\CsSS$-\emph{equivalence} if its image under $\cL_{\CsSS}$ is an equivalence.
	\end{definition}
	\begin{rem}
		Because $\cL_{\CsSS}$ preserves colimits, the collection of $\CsSS$-equivalences is closed under all colimits. Because equivalences in $\CsSS$ have the 2-out-of-3 property, so do $\CsSS$-equivalences. This makes the class of $\CsSS$-equivalences strongly saturated. By a small object argument (see for example \cite{mazelgee2025model} 3.6), the class of $\CsSS$-equivalences is in fact the smallest strongly saturated class containing $J_{JL}$.  
	\end{rem}
	As in \cite{harpaz2015quasi}, let
	\begin{align}\label{eq:marked_forgetful}
		\cF_{+} \colon s\cS \to ss\cS_{+}
	\end{align}
	denote the \textit{marked forgetful functor}, that takes a simplicial set to the semi-simplicial set $\cF(X)$ with the marking determined by (the image of) $s_0\colon X_0 \to X_1$. The main theorem of \cite{harpaz2015quasi} says that $\cF_{+}$ and its right adjoint $\cR_{+}$ give an equivalence of categories
	\begin{align}\label{eq:CSS=CsSS}
		\cL_{\CsSS}\circ \cF_{+}: \CSS \leftrightarrows \CsSS : \cR_{+} \, ,
	\end{align}
	where $\CSS$ denotes the full subcategory of complete Segal spaces, which is again equivalent to $\Cat_\infty$.

	We now set up a Joyal--Tierney calculus for marked bi-semi-simplicial spaces, which we use to give an explicit formula for the restricted localization
	\begin{align}\label{eq:restricted_L_CSS}
		\cL_{\CsSS}\colon \IK_{+} \to \CsSS\, .
	\end{align}
	The techniques used here are a generalization of those used in \cite{joyal2007quasi}. Define the category of \emph{marked bi-semi-simplicial spaces} to be
	\begin{align*}
		bss\cS_{+} &= \Fun(\Simp^{\op}_s, ss\cS_{+})\, .
	\end{align*} 
	The opposite of the Yoneda embedding $\Simp \to ss\cS$ is the free completion of $\Simp^{\op}$, so restriction determines an equivalence
	\begin{align}\label{eq:Fun*R_restriction}
		\Fun^{R}(ss\cS^{\op}, ss\cS_{+}) \xrightarrow{\sim} bss\cS_{+},
	\end{align}
	where $\Fun^{R}(\cC,\cD)\subset \Fun(\cC,\cD)$ denotes the full subcategory of limit preserving (right adjoint) functors. Right Kan extension gives an inverse of \cref{eq:Fun*R_restriction}, and we comppose to get a functor
	\begin{align}\label{eq:slash_construction}
		bss\cS_{+} \xrightarrow{\sim} \Fun^{R}(ss\cS^{\op}, ss\cS_{+}) \hookrightarrow \Fun(ss\cS^{\op}, ss\cS_{+}).
	\end{align}
	For an object $Z\in bss\cS_{+}$, write
	\begin{align*}
		Z/- \colon ss\cS^{\op} \to ss\cS_{+}
	\end{align*}
	for the value of  $Z$ under \cref{eq:slash_construction}. By construction, each $Z/-$ admits a left adjoint, and we write
	\begin{align*}
		-\backslash Z \colon ss\cS_{+} \to ss\cS^{\op}
	\end{align*}
	for the opposite of this left adjoint. Both $Z/-$ and $-\backslash Z$ map colimits to limits. 
	\begin{rem}\label{rem:box_product}
		For a fixed $X\in ss\cS_{+}$, the term  $X\backslash Z$ is covariantly natural in $Z$. As a functor $bss\cS_{+} \to ss\cS$ it can be described as the composite
		\begin{align}\label{eq:X_slash-}
			X\backslash- \colon bss\cS_{+} \xrightarrow{\sim} \Fun^{R}(ss\cS^{\op},ss\cS_{+}) \hookrightarrow \Fun(ss\cS^{\op},ss\cS_{+}) \xrightarrow{ev_{X}} ss\cS_{+} \, .
		\end{align}
		Each of the components preserves limits, so by coherently picking left adjoints to the functors \cref{eq:X_slash-}, one can produce a bifunctor
		\begin{align*}
			-\square - \colon ss\cS_{+} \timest ss\cS \to bss\cS
		\end{align*}
		preserving colimits in each variable, uniquely determined by the functor $X\backslash -$ being right adjoint to $X\square -$ and $-/Y$ being right adjoint to $-\square Y$. 
	\end{rem}
	\begin{rem}\label{rem:J_slash_Z_identification}
		Essentially by definition, we have that for $Z\in bss\cS_{+}$, the marked semi-simplicial space $Z/\Delta^{m}_{s}$ is the value of $Z$ at $[m] \in \Simp^{\op}_{s}$. Using the adjunction defining $-\backslash Z$ we have
		\begin{align*}
			\Map_{ss\cS}(\Delta^{m}_{s}, J\backslash Z) &\simeq \Map_{ss\cS_{+}}(J, Z / \Delta^{m}_{s} ) \\
			&\simeq \Map_{ss\cS_{+}}(J, Z([m])) \, . 
		\end{align*}
		This means that the bifunctor $\square$ can be uniquely characterized preserving colimits in each variable separately, 
		\begin{align*}
			\Map(\Delta^{n,\flat}_{s} \square \Delta^{m}_{s}, Z) \simeq Z([m])([n])
		\end{align*}
		and $\Map(\Delta^{1,\sharp}_{s} \square \Delta^{m}, Z)$ being the subspace of marked edges in $Z([m])([1])$.
	\end{rem}

	\begin{definition}\label{def:vertical_fibration}
		We say that a map of marked bi-semi-simplicial spaces $Z\to W$ is a \emph{vertical fibration} if for any map $j\colon J \to K \in I_{JL}$, the induced map
		\begin{align}\label{eq:vertical_fibration}
			K\backslash Z \to K \backslash W \timest_{J\backslash W} J\backslash Z 
		\end{align}
		is a Kan fibration of semi-simplicial spaces. We say that $Z$ is \emph{vertically fibrant} if the unique map to the terminal object is a vertical fibration. 
	\end{definition}
	\begin{rem}\label{rem:vertically_fibrant}
		Because $K\backslash -$ and $J\backslash -$ preserve limits, and in particular take the terminal object to the terminal object, we have that $Z$ is vertically fibrant if and only if $K\backslash Z \to J\backslash Z$ is a Kan fibration for all $J\to K \in I_{JL}$. 
	\end{rem}
	\begin{rem}\label{rem:vertically_fibrant_weakly_saturated}
		Because $-\backslash Z$ and $-\backslash W$ take colimits to limits, and because the class of Kan fibrations is closed under pullbacks, $X\to Y$ is a vertical fibration if and only if \cref{eq:vertical_fibration} is a Kan fibration for $j\colon J\to K$ any marked cofibration.
	\end{rem}
	Consider the functor
	\begin{align*}
		p^{*}\colon ss\cS_{+} = \Fun([0], ss\cS_{+}) \to \Fun(\Simp^{\op}_{s}, ss\cS_{+})
	\end{align*}
	determined by precomposition with the unique functor $\Simp^{\op}_{s}\to[0] $. This admits a left adjoint $p_{!}$ which takes a semi-simplicial object in $ss\cS_{+}$ to its colimit. In other words, $p_!Z$ is the realization of the semi-simplicial object $Z(\Delta^{\bullet}_{s}) = Z/\Delta^{\bullet}_{s}$. Using the adjunctions and the fact that mapping out of $\Delta^{n,\flat}$ preserves colimits in $ss\cS_{+}$, we get
	\begin{align}\label{eq:p!Z_n}
		\Map(\Delta^{n,\flat}_{s}, p_!Z) &\simeq \Vert \Map(\Delta^{n,\flat}_{s}, Z/\Delta^{\bullet}_{s}) \Vert \\
		&\simeq \Vert \Map(\Delta^{\bullet}_{s}, \Delta^{n,\flat}_{s}\backslash Z)\Vert  \nonumber \\
		&\simeq \Vert \Delta^{n,\flat}_{s}\backslash Z \Vert \,. \nonumber
	\end{align}
	The marking of $p_!Z$ is determined by the image of the colimit of marked edges in each $Z(\Delta^{m}_{s})$, so by a similar argument to the one above, $p_!Z$ is marked by the image of
	\begin{align}\label{eq:p!Z_marking}
		\Vert \Delta^{1,\sharp} \backslash Z \Vert \to \Vert \Delta^{1,\flat}_{s}\backslash Z \Vert \, .
	\end{align}
	The point of the vertical fibrancy condition is that it gives us enough Kan fibrations to commute limits with geometric realizations and therefore control $p_!Z$. Consider the presheaves $\Map(-, p_{!}Z)$ and $\Vert -\backslash Z \Vert$ on $ss\Set_{+}$. When restricted to the full subcategory $\cC \subset ss\Set_{+}$ spanned by $\Delta^{n,\flat}_{s}$ for $n\geq 0$ and $\Delta^{1,\sharp}_{s}$, we have a natural transformation
	\begin{align*}
		\alpha\colon \Vert - \backslash Z \Vert \to p_{!}Z(-)
	\end{align*}
	of presheaves on $\cC$ given by \cref{eq:p!Z_n} and \cref{eq:p!Z_marking}. Since $p_{!}Z(-)$ is right Kan extended from its restriction to $\cC$, we can uniquely extend $\alpha$ to a natural transformation of presheaves on $ss\Set_{+}$. 
	
	\begin{lem}\label{lem:p!Z}
		If $Z$ is a vertically fibrant bi-semi-simplicial space, and $J$ is a semi-simplicial set, then the natural transformation constructed above gives an equivalence
		\begin{align*}
			p_!Z(J) = \Vert J\backslash Z \Vert.
		\end{align*}
	\end{lem}
	\begin{proof}
		The functor $J\mapsto p_!Z(J)$ maps colimits to limits. By vertical fibrancy, \cref{lem:semi_kan_realization} and \cref{lem:inverse_lim_realization}, $J\mapsto \Vert J\backslash Z\Vert$ takes pushouts where one leg is a marked cofibration to pullbacks, and sequential colimits over marked cofibrations to inverse limits. The category of marked semi-simplicial sets is generated by the objects $\Delta^{n,\flat}_{s}$ and $\Delta^{1,\sharp}$ under such colimits. From \cref{eq:p!Z_n} the presheaves agree on $\Delta^{n,\flat}_{s}$ so only $\Delta^{1,\sharp}_{s}$ remains. Here it suffices to show that \cref{eq:p!Z_marking} is an equivalence onto its image, i.e. a monomorphism.
		
		The map $\Delta^{1,\sharp}_{s}\backslash Z \to \Delta^{1,\flat}_{s} \backslash Z$ is a Kan fibration between semi-Kan spaces by the vertical fibrancy of $Z$, and a levelwise monomorphism because the map at level $m$ is the inclusion of the marked edges in $Z/\Delta^{m}_{s}$. By \cref{lem:unit_is_anodyne}, we can extend any map $\Delta^{0}_{s} \to \Delta^{1,\flat}_{S} \backslash Z$ along $\Delta^{0}_{s} \to T$. We can then use \cref{lem:semi_kan_realization} to compute any fiber of \cref{eq:p!Z_marking} as the geometric realization of a Kan space
		\begin{align*}
			F= T\timest_{\Delta^{1,\flat}_{s}\backslash Z}  \Delta^{1,\sharp}_{s} \backslash Z \, ,
		\end{align*}
		which has each $F_{n} = \emptyset$ or $\ast$. As $F$ is Kan, this can only happen if $F= \emptyset$ or $T$. 
	\end{proof}
	\begin{lem}\label{lem:t*_characterization}
		There exists a limit-preserving functor 
		\begin{align*}
			t^{*}\colon ss\cS_{+} \to bss\cS_{+}
		\end{align*}
		uniquely determined by natural equivalences
		\begin{align*}
			J\backslash t^{*}X \simeq \widetilde{U}(X^{J}) \, .
		\end{align*}
	\end{lem}
	\begin{proof}
		Consider the functor
		\begin{align*}
			ss\cS_{+}\timest ss\cS_{+}^{\op} & \to ss\cS \\
			(X,J) & \mapsto \widetilde{U}(X^{J}) \, .
		\end{align*}
		Because $\widetilde{U}$ is a right adjoint, this preserves limits in each variable separately, so it corresponds under adjunction to a limit-preserving functor
		\begin{align*}
			ss\cS_{+} \to \Fun^{R}(ss\cS_{+}^{\op}, ss\cS) \, .
		\end{align*}
		The target is equivalent to $bss\cS_{+}$ with inverse given by $Z \mapsto -\backslash Z$. 
	\end{proof}
	
	The following lemma is the key to relating right lifting properties and orthogonality.
	\begin{lem}\label{lem:marked_inner_anodyne_stability}
		If $A\to B$ is a marked inner anodyne map of marked semi-simplicial sets, and $J\to K$ is a marked cofibration between marked semi-simplicial sets, then the induced map
		\begin{align*}
			A\otimes K \amalga_{A\otimes J} B\otimes J \to B\otimes K
		\end{align*}
		is also marked inner anodyne.
	\end{lem}
	\begin{proof}
		This result holds in the category $ss\Set_+$ by \cite{henry2018weak} 5.3.10 and 5.5.6. By \cref{rem:cof_weakly_saturated}, any marked inner anodyne between marked semi-simplicial sets is levelwise injective. Pushouts in $ss\Set_{+}$ where one leg is levelwise injective are also pushouts when considered as diagrams in $ss\cS_{+}$, and so the result carries over to $ss\cS_{+}$.
	\end{proof}
	\begin{cor}\label{cor:marked_fibration_stability}
		If $X\to Y$ is a marked inner fibration, and $j\colon J\to K$ a marked cofibration between marked semi-simplicial sets, then the induced map
		\begin{align}
			f\colon X^{K} \to X^{J}\timest_{Y^{J}} Y^{K}
		\end{align}
		is also a marked inner fibration. Moreover, if $j$ is marked inner anodyne, then $f$ is a trivial fibration. 
	\end{cor}
	\begin{proof}
		Let $s\colon S\to T$ be any marked inner anodyne map of semi-simplicial sets. By adjunction, it suffices to show that $X\to Y$ has the right lifting property with respect to
		\begin{align*}
			S\otimes K \amalga_{S\otimes J} T\otimes J \to T\otimes K
		\end{align*}
		which is marked inner anodyne by \cref{lem:marked_inner_anodyne_stability}. If $j$ is marked inner anodyne, we can take $s$ to be any monomorphism.
	\end{proof}
	\begin{cor}\label{cor:tildeU_fibrations}
		If $X\to Y$ is a marked inner fibration, and $j\colon J\to K$ is a marked cofibration between marked semi-simplicial sets, then the induced map
		\begin{align*}
			f\colon \widetilde{U}(X^{K}) \to \widetilde{U}(X^{J} \timest_{Y^{J}} Y^{K})
		\end{align*}
		is a Kan fibration. Moreover, if $j$ is marked inner anodyne, then $f$ is a trivial fibration.
	\end{cor}
	\begin{proof}
		By the proofs of \cite{harpaz2015quasi} 2.1.7 and 2.1.8, the maps $\Horn^{n,\sharp}_{i,s} \to \Delta^{n,\sharp}_{i}$ are marked inner anodyne, so by adjunction $\widetilde{U}$ takes marked inner fibrations to Kan fibrations. The maps $\del \Delta^{n,\sharp} \to \Delta^{n, \sharp}$ are marked cofibrations, so by adjunction, $\widetilde{U}$ takes trivial fibrations to trivial fibrations. The result then follows from \cref{cor:marked_fibration_stability}.
	\end{proof}

	\begin{cor}\label{lem:inner_kan_fibrant}
		If $X$ is a marked inner Kan space, then the marked bi-semi-simplicial space $t^*X$ is vertically fibrant.
	\end{cor}
	\begin{proof}
		By the universal property defining $t^{*}X$ it suffices to show that for any marked cofibration $J\to K$ between semi-simplicial sets, the map
		\begin{align*}
			\widetilde{U}(X^{K}) &\to \widetilde{U}(X^{J})\\
		\end{align*}
		is a Kan fibration. This follows immediately from \cref{cor:tildeU_fibrations}.
	\end{proof}
	Consider the composite functor
	\begin{align*}
		\cT \colon \IK_{+} \subset ss\cS_{+} \xrightarrow{t^{*}} bss\cS_{+} \xrightarrow{p_{!}} ss\cS_{+} \, .
	\end{align*}
	We will show that $\cT$ is equivalent to the restricted localization \cref{eq:restricted_L_CSS}.
	
	\begin{cor}\label{cor:T(X)(J)}
		Let $X$ be a marked inner Kan space, and $J$ a marked semi-simplicial set. Then we have a natural equivalence
		\begin{align*}
			\cT(X)(J) \simeq \Vert \widetilde{U}(X^{J}) \Vert \, .
		\end{align*}
	\end{cor}
	\begin{proof}
		By \cref{lem:inner_kan_fibrant}, the marked bi-semi-simplicial space $t^{*}X$ is vertically fibrant. For any marked semi-simplicial set $J$, we have by \cref{lem:t*_characterization} that
		\begin{align*}
			\cT(X)(J) \simeq p_{!}t^{*}X(J) \simeq \Vert J\backslash t^{*}X \Vert \simeq \Vert \widetilde{U}(X^{J}) \Vert\, .
		\end{align*}
	\end{proof}
	\begin{cor}\label{lem:inner_fibration_pullback}
		Assume that in a pullback square of marked semi-simplicial spaces
		\begin{equation*} \cQ =
			\begin{tikzcd}
				W \arrow[d] \arrow[r]& X \arrow[d, "f"] \\
				Z \arrow[r] &  Y \, ,
			\end{tikzcd}
		\end{equation*}
		the objects are marked inner Kan spaces, and the map $f$ is a marked inner fibration. Then the square $\cT(\cQ)$ is also a pullback.
	\end{cor}
	\begin{proof}
		It suffices to check that we get a pullback square $\cT(\cQ)(J)$ when we evaluate at $J$ any marked simplex. By \cref{cor:T(X)(J)},  $\cT(\cQ)(J)$ is equivalent to the geometric realization of the following pullback square.
		\begin{equation*}
			\begin{tikzcd}
				\widetilde{U}(W^J) \arrow[r] \arrow[d] & \widetilde{U}(X^{J}) \arrow[d] \\
				\widetilde{U}(Z^{J}) \arrow[r] & \widetilde{U}(Y^{J})
			\end{tikzcd}
		\end{equation*}
		By \cref{cor:tildeU_fibrations}, the objects in this diagram are semi-Kan spaces, and the vertical maps are Kan fibrations. The pullback is therefore preserved under geometric realization by \cref{lem:semi_kan_realization}.
	\end{proof}
	\begin{cor}\label{cor:inverse_limit_marked}
		Let $\cI$ be an inverse category, and $F\colon \cI \to ss\cS_{+}$ a $J_{JL}$-Reedy fibrant diagram. Then the terms $F(i)$, as well as the limit $\lim(F)$ are marked inner Kan spaces, and the induced map
		\begin{align}\label{eq:inverse_limit_marked_1}
			\cT(\lim(F)) \to \lim(\cT(F))
		\end{align}
		is an equivalence.
	\end{cor}
	\begin{proof}
		Each $F(i)$ is the limit of the restriction of $F$ to $\cI_{\leq i}$, and one can check that $F_{\leq i} \to \underline{T}$ is a $J_{JL}$-Reedy fibration, so by \cref{lem:inverse_limit_preserves_fib}, the induced map on limits $F(i)\to T$ is a marked inner fibration. This also works for $i = \infty$ to get that $\lim F \to T$ is a marked inner fibration.
		
		To show that \cref{eq:inverse_limit_marked_1} is an equivalence, it suffices to show that it is an equivalence when evaluating at any marked semi-simplicial space $J$. Because evaluation and exponentiation preserve limits, we can use \cref{cor:T(X)(J)} to identify the map in question with
		\begin{align}\label{eq:inverse_limit_marked_2}
			\Vert \widetilde{U}(\lim F^{J}) \Vert \to \lim \Vert \widetilde{U}(F^{J}) \Vert \, .
		\end{align}
		Because exponentiation and $\widetilde{U}$ preserve limits, the matching map at $i$ of the diagram $\widetilde{U}(F^{J})\colon \cI \to ss\cS$ can be identified with
		\begin{align*}
			\widetilde{U}(F(i)^{J}) \to \widetilde{U}(M_{F}(i)^{J}) \, ,
		\end{align*}
		which is a Kan fibration by \cref{cor:tildeU_fibrations} and the $J_{JL}$-Reedy fibrancy of $F$. This shows that $\widetilde{U}(F^{J})$ is $J_{KQ}$-Reedy fibrant, so by \cref{lem:inverse_lim_realization}, the map \cref{eq:inverse_limit_marked_2} is an equivalence.
	\end{proof}
	Note that for any marked semi-simplicial set $J$, we have
	\begin{align*}
		\Map(J, t^{*}X(\Delta^{n}_{s})) &\simeq \Map(\Delta^{n}_{s}, J\backslash t^{*}X) \\
		&\simeq \Map(\Delta^{n}_{s}, \widetilde{U}(X^{J}) ) \\
		&\simeq \Map(\Delta^{n,\sharp}_{s} \otimes J, X) \\
		&\simeq \Map(J,X^{\Delta^{n,\sharp}_{s}}) \, .
	\end{align*}
	These equivalences are furthermore natural in $J$, so we can think of $t^{*}X$ as the semi-simplicial object $X^{\Delta^{\bullet,\sharp}_{s}}$ in $ss\cS_{+}$, and $p_{!}t^{*}X$ as the realization of this object. The zeroeth level of this object is $X^{\Delta^{0}_{s}}\simeq X$, so by including into the colimit we get a map
	\begin{align*}
		u_{X}\colon X\to \cT(X) \, .
	\end{align*}
	This construction is also natural in $X$, meaning that it extends to a natural transformation $u\colon id\to \cT$. 
	\begin{lem}\label{lem:u_is_trivial_fib}
		For a marked inner Kan space $X$, the map $u_{X}\colon X\to \cT(X)$ is a trivial Kan fibration.
	\end{lem}
	\begin{proof}
		Let $J\to K$ be a marked cofibration between marked semi-simplicial sets. By picking a base point in each component of $X(J)$, we get a section of the natural projection $X(J) \to \pi_{0}X(J)$, which is by definition an epimorphism. Consider the following diagram of spaces defined by all squares being pullbacks.
		\begin{equation}\label{eq:U_is_trivial_fib_0}
			\begin{tikzcd}
				X(K)' \arrow[d] \arrow[r] & X(K) \arrow[d] \\ 
				P'  \arrow[r] \arrow[d]      &   P \arrow[d]  \arrow[r] & \cT(X)(K) \arrow[d] \\
				\pi_{0}X(J)  \arrow[r]               &   X(J)  \arrow[r]            & \cT(X)(J)
			\end{tikzcd}
		\end{equation} 
		To prove the result, we must show that $X(K) \to P$ is an epimorphism. Because epimorphisms of spaces are preserved under base change, the maps $P'\to P$ and $X(K)'\to X(K)$ are epimorphisms, so by 2-out-of-3, it suffices to show that $X(K)' \to P'$ is an epimorphism. By \cref{cor:T(X)(J)}, $P'$ is equivalent to the pullback 
		\begin{align*}
			\pi_{0}X(J) \timest_{\Vert \widetilde{U}(X^{J}) \Vert} \Vert \widetilde{U}(X^{K}) \Vert \, .
		\end{align*}
		By construction of $u$, we may think of $\pi_{0}X(J)$ as a semi-simplicial set concentrated in degree zero, and the map used in forming the above pullback as the realization of the map
		\begin{align}\label{eq:U_is_trivial_fib_1}
			\pi_{0}X(J) \to X(J) \to \widetilde{U}(X^{J}) \, ,
		\end{align}
		where the second factor is the inclusion of 0-simplices. The constant semi-simplicial set $\underline{\pi_{0}X(J)}$ is equivalently $\cF\cL(\pi_{0}X(J))$, and so by \cref{lem:unit_is_anodyne} and $\widetilde{U}(X^{J})$ being an inner Kan space, we may extend \cref{eq:U_is_trivial_fib_1} to a map
		\begin{align*}
			\underline{\pi_{0}X(J)} \to \widetilde{U}(X^{J}) \, ,
		\end{align*}
		whose realization produces a map equivalent to \cref{eq:U_is_trivial_fib_1}. Because $\underline{\pi_{0}X(J)}$ is a semi-Kan space and $\widetilde{U}(X^{K}) \to \widetilde{U}(X^{J})$ a Kan fibration betweem semi-Kan spaces, it follows from \cref{lem:semi_kan_realization} that the pullback $\underline{\pi_{0}X(J)} \timest_{\widetilde{U}(X^{J})} \widetilde{U}(X^{K})$ is preserved under realization. We may therefore identify the map $X(K)' \to P'$ with the realization of
		\begin{align}\label{eq:u_is_trivial_fib_2}
			\underline{X(K)}' \to \underline{\pi_{0}X(J)} \timest_{\widetilde{U}(X^{J})} \widetilde{U}(X^{K}) \, .
		\end{align}
		To prove that a map of semi-simplicial spaces becomes an epimorphism after geometric realization, it suffices to prove that the map on $0$-simplicices is an epimorphism. On $0$-simplices, \cref{eq:u_is_trivial_fib_2} is given by 
		\begin{align*}
			X(K)' \to \pi_{0}X(J)\timest_{X(J)} X(K) \, , 
		\end{align*}
		which is an equivalence because the left rectangle of \cref{eq:U_is_trivial_fib_0} is a pullback.
	\end{proof}
	
	\begin{definition}\label{def:delta1_homotopy}
		A $\Delta^{1,\sharp}_{s}$-\emph{homotopy} between two maps $f,g\colon X\to Y$ of marked semi-simplicicial spaces is a map $H\colon X \otimes \Delta^{1, \sharp}_{s}\to Y$  which restricts along $id \otimes d_i \colon X\otimes \Delta^{0}_{s} \to X\otimes \Delta^{1, \sharp}_{s}$ to $f$ and $g$, respectively.
	\end{definition}
	\begin{definition}\label{def:delta1_homotopy_eq}
		A map $f\colon X\to Y$ of marked semi-simplicial spaces is a $\Delta^{1,\sharp}_{s}$-\emph{homotopy equivalence} if there exists a map $g\colon Y\to X$ and $\Delta^{1,\sharp}_{s}$-homotopies $g\circ f \sim id_{X}$ and $f\circ g \sim id_{Y}$. 
	\end{definition}
	By \cref{lem:marked_inner_anodyne_stability}, either inclusion $id\otimes d_{i}\colon X \otimes \Delta^{0}_{s} \to X\otimes \Delta^{1,\sharp}_{s}$ is marked inner anodyne, and so in particular a $\CsSS$-equivalence when $X$ is levelwise discrete. An arbitrary $X$ is a colimit of levelwise discrete objects, and since $\otimes$ preserves colimits in each variable, $X\to X\otimes \Delta^{1,\sharp}_{s}$ is always a $\CsSS$-equivalence. It follows that any $\Delta^{1,\sharp}_{s}$-homotopy equivalence is also a $\CsSS$-equivalence.
	\begin{lem}\label{lem:trivial_fibration=>eq}
		A trivial fibration $f\colon X\to Y$ between marked inner Kan spaces is a $\CsSS$-equivalence.
	\end{lem}
	\begin{proof}
		First, assume that $X$ and $Y$ are levelwise discrete. Then the map $f\colon X\to Y$ is a trivial fibration between bifibrant objects in the weak Joyal--Lurie model structure. By \cite{henry2018weak}, trivial fibrations are also homotopy equivalences, and homotopy in the weak Joyal--Lurie model structure is precisely $\Delta^{1,\sharp}_{s}$-homotopy. As remarked above, $\Delta^{1,\sharp}_{s}$-homotopy equivalences are $\CsSS$-equivalences.

		If $X$ and $Y$ are arbitrary marked inner Kan spaces, the map on underlying semi-simplicial spaces $U(X)\to U(Y)$ can be modelled as a vertical fibration between vertically fibrant objects in $\Fun(\Delta_{s}^{\op}, s\Set)$. By forgetting degeneracies in the target, this produces a map $U(\overline{f})\colon U(\overline{X}) \to U(\overline{Y})$ of bi-semi-simplicial sets. Using the adjunction
		\begin{align*}
			\Map(\Delta^{1}_{s}, U(\overline{X})/\Delta^{m}_{s} ) \simeq \Map(\Delta^{m}, \Delta^{1}_{s}\backslash U(\overline{X}) ) \, ,
		\end{align*}
		we can mark the semi-simplicial set $U(\overline{X})/\Delta^{m}_{s}$ by marking all the simplices 
		\begin{align*}
			\Delta^{m}_{s} \to \Delta^{1}_{s}\backslash U(\overline{X})
		\end{align*}
		whose induced map on geometric realizations
		\begin{align*}
			\Vert \Delta^{m}_{s} \Vert \to \Vert \Delta^{1}_{s}\backslash U(\overline{X}) \Vert \simeq X(\Delta^{1,\flat}_{s})
		\end{align*}
		factors through the marking of $X$. Doing the same for $Y$ produces a map $\overline{f}\colon \overline{X} \to \overline{Y}$ of marked bi-semi-simplicial sets such that $p_{!}\overline{f} \simeq f$. We claim moreover that $\overline{f}$ is a vertical fibration between vertically fibrant objects. Because we assumed that $U(\overline{f})$ is a vertical fibration, it only remains to check that the maps
		\begin{align*}
			\Delta^{1,\sharp}_{s}\backslash \overline{Y} & \to \Delta^{1,\flat}_{s}\backslash \overline{Y}, \\
			\Delta^{1,\sharp}_{s} \backslash \overline{X} &\to \Delta^{1,\sharp}_{s} \backslash \overline{Y} \timest_{\Delta^{1,\flat}_{s}\backslash \overline{Y}} \Delta^{1,\flat}_{s} \backslash \overline{X} \, ,
		\end{align*}
		are Kan fibrations. This follows because we defined the marking in terms of geometric realization, and horn inclusions become equivalences on geometric realization. 
		
		Let $J\to K$ be a marked inner anodyne map between marked semi-simplicial sets. Because $Y$ is a marked inner Kan space, the map
		\begin{align*}
			Y(K) \to Y(J)
		\end{align*}
		is surjective on connected components. Because $\overline{Y}$ is vertically fibrant, we can use \cref{lem:p!Z} to write this as the geometric realization of the map
		\begin{align}\label{eq:trivial_fib_is_eq}
			K\backslash \overline{Y} \to J\backslash \overline{Y},
		\end{align}
		which is moreover a Kan fibration between Kan semi-simplicial sets. Since \cref{eq:trivial_fib_is_eq} is a Kan fibration and induces surjection on connected components, it must necessarily be surjective on $0$-cells. Because the inclusion of the zeroeth vertex $\Delta^{0}_{s} \to \Delta^{m}_{s}$ is anodyne, it follows by Kan fibrancy that \cref{eq:trivial_fib_is_eq} is also surjective on $m$-cells for all $m$. By adjunction, this implies that each row $Y/\Delta^{m}_{s}$ is marked inner Kan.
		
		Now let $J\to K$ be any marked cofibration between marked semi-simplicial sets. Because $f$ is a trivial fibration, the map
		\begin{align}\label{eq:trivial_fib_is_eq_2}
			X(K) \to X(J) \timest_{Y(J)} Y(K)
		\end{align}
		is surjective on connected components. Because $\overline{f}$ is a vertical fibration, we can use \cref{lem:p!Z,lem:semi_kan_realization} to write \cref{eq:trivial_fib_is_eq_2} as the geometric realization of the Kan fibration
		\begin{align*}
			K\backslash \overline{X} \to J\backslash \overline{X} \timest_{J\backslash \overline{Y}} K\backslash \overline{Y}.
		\end{align*}
		Arguing as above, this map is surjective on $m$-simplices for all $m$, so each row
		\begin{align*}
			\overline{f}/\Delta^{m}_{s} \colon \overline{X}/\Delta^{m}_{s} \to \overline{Y}/\Delta^{m}_{s}
		\end{align*}
		is a trivial fibration between levelwise discrete marked inner Kan spaces, and therefore a $\CsSS$ equivalence by the first part of this proof. Hence $f\simeq p_! \overline{f}$ is a colimit of $\CsSS$ equivalences, and therefore itself a $\CsSS$-equivalence.
	\end{proof}
	
	\begin{cor}\label{cor:T_is_L}
		For a marked inner Kan space $X$, the map $u_{X}\colon X\to \cT(X)$ factors uniquely through an equivalence $\widetilde{u}_{X} \colon \cL_{\CsSS}(X) \to \cT(X)$.
	\end{cor}
	\begin{proof}
		The factorization exists because $\cT(X)$ is a complete semi-Segal space. The map $\widetilde{u}_{X}$ is a $\CsSS$ equivalence by 2-out-of-3 because $X\to \cT(X)$ is a trival Kan fibration and therefore a $\CsSS$ equivalence by \cref{lem:trivial_fibration=>eq}.
	\end{proof}
	Because every marked complete semi-Segal space is also a marked inner Kan space, we get a restricted localization
	\begin{align*}
		\cT :  \IK_{+} \leftrightarrows \CsSS : i_{\CsSS}.
	\end{align*}
	\begin{cor}\label{cor:T_characterization}
		The restricted functor $\cT\colon \IK_{+} \to \CsSS$ exhibits a localization at trivial fibrations.
	\end{cor}
	\begin{proof}
		Write $\cL_{w}\colon \IK_{+} \to \IK_{+}[w^{-1}]$ for the localization at trivial fibrations. By \cref{lem:trivial_fibration=>eq}, trivial fibrations between inner Kan spaces are $\CsSS$ equivalences, so by the universal property of localizations, the restriction of $\cL_{\CsSS}$ factors as 
		\begin{align*}
			\IK_{+} \xrightarrow{\cL_{w}} \IK_{+}[w^{-1}] \xrightarrow{F} \CsSS \, .
		\end{align*}
		Now we claim that the composite 
		\begin{align*}
			G\colon \CsSS \xhookrightarrow{i_{\CsSS}} \IK_{+} \xrightarrow{\cL_{w}} \IK_{+}[w^{-1}]
		\end{align*}
		is an inverse of $F$. The relation $F\circ G\sim id$ follows immedeatly from $\cL_{\CsSS} \circ i_{\CsSS} \sim id$. For the other direction, consider the following commutative diagram.
		\begin{equation*}
			\begin{tikzcd}
				\IK_{+} \arrow[r, "\cT"] \arrow[d, "\cL_{w}"]      &   \CsSS  \arrow[d, "id"]  \arrow[r, "i_{\CsSS}"] & \IK_{+} \arrow[d, "\cL_{w}"] \\
				\IK_{+}[w^{-1}]  \arrow[r, "F"]               &   \CsSS  \arrow[r, "G"]            & \IK_{+}[w^{-1}]
			\end{tikzcd}
		\end{equation*}
		Because the natural transformation $u\colon id \to \cT$ has components that are trivial fibrations, the whiskered transformation 
		\begin{align*}
			\cL_{w}(u)\colon \cL_{w} \to \cL_{w} \circ \cT \simeq G\circ F \circ \cL_{w}
		\end{align*}
		is a natural equivalence. It then follows from the uniqueness part of the universal property of localizations that $G\circ F \sim id$. 
	\end{proof}

	\section{Slices and functor categories}\label{sec:slices_and_functor_categories}
	Thanks to the description of $\cT$ as a localization at fibrations, we can describe how $\cT$ behaves with respect to certain right-adjoint constructions. In particular, we will treat exponential objects and slices. The case of slices is of particular interest to applications to flow categories \cite{HP} because slices govern limits and colimits. A central part of proving that $\infty$-categories of flow categories are stable is identifying the most basic limits and colimits.
	
	In this section, we observe a typical feature of weak model structures: constructions are well-behaved as long as we map out of cofibrant (levelwise discrete) objects into fibrant objects (marked inner Kan spaces).
	
	\begin{cor}\label{cor:T(X)^J}
		Let $X$ be a marked inner Kan space, and $J$ a marked semi-simplicial set. Then the natural map
		\begin{align*}
			\cT(X^{J}) \to \cT(X)^{J}
		\end{align*}
		is an equivalence.
	\end{cor}
	\begin{proof}
		The natural map comes from a $(\colim \lim \to \lim \colim)$ comparison. Evaluating at any marked semi-simplicial set $S$, we get
		\begin{align*}
			\cT(X^{J})(S) &\simeq \Vert \widetilde{U}( (X^{J})^{S}) \Vert \\
			&\simeq \Vert \widetilde{U}( X^{J\otimes S} ) \Vert \\
			&\simeq \cT(X)(J\otimes S) \\
			&\simeq \cT(X)^{J}(S) \, .
		\end{align*}
		Here, we use the fact that $X^{J}$ is a marked inner Kan space by \cref{cor:marked_fibration_stability} and then apply \cref{cor:T(X)(J)} twice.
	\end{proof}
	\begin{lem}\label{lem:CsSS_exponentiation}
		Let $W$ be a marked complete semi-Segal space, and $J\to K$ a $\CsSS$-equivalence. Then restriction induces an equivalence $W^{K}\to W^{J}$ between marked complete semi-Segal spaces.
	\end{lem}
	\begin{proof}
		It follows by adjunction from \cref{lem:marked_inner_anodyne_stability} that $W^{K}\to W^{J}$ is orthogonal to $J_{JL}$ whenever $J\to K$ is a marked cofibration between levelwise discrete objects. Applying this to $\emptyset \to J$ implies that $W^{J}$ is a complete marked semi-Segal space whenever $J$ is levelwise discrete. Because $\CsSS$ is closed under limits, and the levelwise discrete objects generate $ss\cS_{+}$ under colimits, and $W^{(-)}$ takes colimits to limits, this shows that any $W^{J}$ is a marked complete semi-Segal space.
		
		Similarly, when $J\to K \in J_{JL}$, it follows by adjunction from \cref{lem:marked_inner_anodyne_stability} that $W^{K}\to W^{J}$ is orthogonal to all marked cofibrations between levelwise discrete objects. Applying this at $\emptyset \to A$ gives the equivalence
		\begin{align*}
			W^{K}(A) \simeq W^{J}(A) \, .
		\end{align*}
		Now the class of maps $J\to K$ such that $W^{K} \to W^{J}$ is an equivalence is strongly saturated and contains $J_{JL}$, and therefore all $\CsSS$-equivalences.
	\end{proof}
	
	Recall that $\cL \colon ss\cS \to s\cS$ is given by left Kan extension along $\Simp_{s} \to \Simp$, while $\cR_{+} \colon ss\cS_{+} \to s\cS$ is the right adjoint to the marked forgetful functor \cref{eq:marked_forgetful} constructed in \cite{harpaz2015quasi}.
	\begin{cor}\label{cor:RT(X)^J}
		If $J$ is a semi-simplicial set and $X$ is a marked inner Kan space, then there is an equivalence
		\begin{align*}
			\cR_{+}\cT\left(X^{J^{\flat}} \right) \to \left(\cR_{+} \cT(X)\right)^{\cL(J)}
		\end{align*}
		of complete Segal spaces. This equivalence is moreover natural in both $X$ and $J$.
	\end{cor}
	\begin{proof}
		Since $X$ is a marked inner Kan space, $\cT(X)$ is a complete semi-Segal space. By \cite{harpaz2015quasi} Lemma 3.4.1, there are natural $\CsSS$-equivalences $J^{\flat} \to \cF_{+} \cL(J)$, so combining \cref{cor:T(X)^J,lem:CsSS_exponentiation} we have equivalences
		\begin{align}\label{eq:RT(X)^J_2}
			\cT(X^{J^{\flat}}) \simeq \cT(X)^{J^{\flat}} \simeq \cT(X)^{\cF_{+}\cL(J)} \, .
		\end{align}
		By \cite{harpaz2015quasi} Lemma 3.4.2, the functor $\cF_{+}$ has a lax monoidal structure where each
		\begin{align*}
			\cF_{+}(Y)\otimes \cF_{+}(Z) \to \cF_{+}(Y\times Z)
		\end{align*}
		is a $\CsSS$ equivalence. Fixing $Z$, we get an induced natural transformation on right adjoints
		\begin{align}\label{eq:RT(X)^J_3}
			\cR_{+}(W)^{Z} \to \cR_{+}(W^{\cF_{+}Z}) \, ,
		\end{align}
		which is an equivalence whenever $W$ is a complete marked semi-Segal space.
		We can therefore apply $\cR_{+}$ to \cref{eq:RT(X)^J_2}, and further compose
		\begin{align*}
			\cR_{+}\cT(J^{\flat}) \simeq \cR_{+}\left(\cT(X)^{\cF_{+}\cL(J)}\right) \simeq \cR_{+}\cT(X)^{\cL(J)},
		\end{align*}
		where the last equivalence is \cref{eq:RT(X)^J_3} for $W=\cT(X)$ and $Z=\cL(J)$.
	\end{proof}
	
	\begin{lem}\label{lem:join_comparison}
		For $J$ and $K$ simplicial sets, there are natural maps
		\begin{align}\label{eq:join_comparison}
			J \altjoin K \to J\join K
		\end{align}
		that are $\CSS$-equivalences. In other words, any complete Segal space is orthogonal to \cref{eq:join_comparison}.
	\end{lem}
	\begin{proof}
		This follows from the proof of \cite{HTT} Lemma 4.2.1.2, with the key point being that inner anodyne morphisms are $\CSS$-equivalences, and the collection of $\CSS$-equivalences is closed under colimits.
	\end{proof}
	Let $J$ be an unmarked semi-simplicial set and $W$ a complete semi-Segal space. Through the chain of maps
	\begin{align*}
		\Map(J^{\flat}, W) \xleftarrow{\simeq} \Map(\cF_{+}\cL(J) ,W) \simeq \Map(\cL(J), \cR_{+}W) \, ,
	\end{align*}
	any  $p\colon J^{\flat} \to W$, determines a map $\overline{p}\colon \cL(J)\to \cR_{+}W$.
	\begin{cor}\label{cor:R_slice_comparison}
		In the above situation, we have a chain of natural equivalences
		\begin{align*}
			\cR_{+}(W_{p/}) \simeq  \cR_{+}(W)_{\overline{p}/} \simeq \cR_{+}(W)^{\tilde{p}/} \, ,
		\end{align*}
		compatible with projection to $\cR_{+}(W)$. 
	\end{cor}
	\begin{proof}
		First note that since the category of complete Segal spaces is closed under forming exponentials and pullbacks, it follows from \cref{lem:alt_undercat_pullback} that $\cR_{+}(W)^{\overline{p}/}$ is also a complete Segal space.
		
		For each $n$, pullback along the $\CSS$-equivalence
		\begin{align*}
			\cL(J) \altjoin \Delta^{n} \to \cL(J) \join \Delta^{n}
		\end{align*}
		determines an equivalence 
		\begin{align}\label{eq:R_slice_comparison}
			(\cR_{+}(W)_{\overline{p}/})(\Delta^{n}) \to (\cR_{+}(W)^{\overline{p}/})(\Delta^{n}) \, .
		\end{align}
		These maps are natural in $\Delta^{n}$ and $J$, so they assemble to an equivalence of simplicial spaces over $\cR_{+}(W)$. 
		
		Similarly, pulling back along the $\CsSS$-equivalences
		\begin{align*}
			J^{\flat}\join\Delta_{s}^{n,\flat} \to \cF_{+}\cL(J\join \Delta_{s}^{n}) \to \cF_{+}(\cL(J) \join\cL(\Delta^{n}_{s})).
		\end{align*}
		induces an equivalence
		\begin{align}\label{eq:R_slice_comparison_2}
			\cF_{+}\left(\cR_{+}(W)_{\overline{p}/} \right) \to W_{p/}
		\end{align}
		covering the counit $\cF_{+}\cR_{+}W \xrightarrow{\sim} W$. Thanks to the equivalence \cref{eq:R_slice_comparison}, $\cR_{+}(W)_{\overline{p}/}$ is a complete Segal space, and by \cite{harpaz2015quasi}, $\cR_{+}$ and $\cF_{+}$ restrict to inverse equivalences between $\CsSS$ and $\CSS$. Therefore the adjunct of \cref{eq:R_slice_comparison_2} is an equivalene $\cR_{+}(W)_{\overline{p}/} \to \cR_{+}(W_{p/})$ covering the identity on $\cR_{+}(W)$.
	\end{proof}
	We now wish to investigate how slices behave with respect to the localization $\cT$. For this purpose, we need yet another kind of fibration.
	\begin{definition}\label{def:marked_left_fibration}
		Let $J_{L}$ be the collection of maps in $ss\cS_{+}$ consisting of $J_{JL}$ and the left horns $\Horn_{0,s}^{n,\flat} \to \Delta^{n, \flat}_{s}$ for $1\leq n$. We call $J_{L}$-fibrations \emph{marked left fibrations}, and $J_{L}$-cofibrations \emph{marked left anodyne}.
		
		Dually, let $J_{R}$ denote the collection of maps consisting of $J_{JL}$ and the right horns $\Horn_{n,s}^{n,\flat} \to \Delta^{n, \flat}_{s}$ for $1\leq n$. We call $J_{R}$-fibrations \emph{marked right fibrations}, and $J_{R}$-cofibrations \emph{marked right anodyne}.
	\end{definition}
	\begin{lem}\label{lem:left_anodyne_stability}
		If $a\colon A\to B$ and $j\colon J\to K$ are marked cofibrations between marked semi-simplicial sets, then the induced map
		\begin{align*}
			f \colon J\join B \underset{J\join A}{\amalga} K\join A \to K\join B
		\end{align*}
		is again a marked cofibration, and:
		\begin{enumerate}
			\item If $a$ is marked left anodyne, then $f$ is marked inner anodyne.
			\item If $j$ is marked left anodyne, then $f$ is also marked left anodyne.
		\end{enumerate}
		The dual statements also hold:
		\begin{enumerate}
			\item If $j$ is marked right anodyne, then $f$ is marked inner anodyne.
			\item If $a$ is marked right anodyne, then $f$ is also marked right anodyne.
		\end{enumerate}
	\end{lem}
	\begin{proof}
		One can see directly from the formula for $\join$ that $f$ is levelwise injective.
		The rest of the proof is a minor modification of \cite{HTT} Lemma 2.1.2.3. As there, we may reduce to the case where $j\in I_{JL}$. If $j$ is $\Delta^{1,\flat}_{s} \to \Delta^{1,\sharp}_{s}$, the map $f$ becomes an equivalence, so it suffices to treat the boundary inclusions $\del \Delta^{n,\flat}_{s} \to \Delta^{n,\flat}_{s}$. As in \cite{HTT}, we may assume that $a$ is one of the generating left anodynes. When $a$ is $\Horn^{n,\flat}_{i} \to \Delta^{n,\flat}_{s}$, where $0\leq i \leq n$, this works just as in the unmarked case. When $a$ is
		\begin{align*}
			(\Horn^{n}_{n}, \left\{ n-1,n\right\}) \to (\Delta_{s}^{n}, \left\{ n-1,n \right\} )
		\end{align*}
		the map $f$ becomes
		\begin{align*}
			(\Horn_{n+m+1,s}^{n+m+1}, \{n+m,n+m+1\}) \to (\Delta^{n+m+1}_{s}, \left\{ n+m, n+m-1\right\} ) \, ,
		\end{align*}
		which is still marked inner anodyne.
		When $a=S_{2/6}$, the map $f$ becomes a re-marking of $\Delta_{s}^{m+3}$, which can easily be written as a pushout of the original $2/6$ map. This finishes the case where $a$ is marked left anodyne. The other cases can be shown in much the same way.
	\end{proof}
	As in \cite{HTT}, this has the following immediate consequences.
	\begin{cor}\label{cor:left_fibration_stability}
		Let $f\colon X\to Y$ be a marked inner fibration, $j\colon J\to K$ a marked cofibration between marked semi-simplicial sets, and $p\colon K\to X$ any map. Then the induced map
		\begin{align*}
			q\colon X_{p/} \to X_{pj/}\timest_{Y_{fpj/}} Y_{fp/}
		\end{align*}
		is a marked left fibration, and:
		\begin{enumerate}
			\item if $j$ is marked right anodyne, then $q$ is a trivial fibration.
			\item if $f$ is a trivial fibration, so is $q$. 
		\end{enumerate}
	\end{cor}
	\begin{proof}
		By adjunction, the lifting problem
		\begin{equation*}
			\begin{tikzcd}
				A \arrow[r] \arrow[d]      &   X_{p/}  \arrow[d] \\
				B  \arrow[r] \arrow[ur,dashed]              &   X_{pj/} \timest_{Y_{fpj/}} Y_{fp/}
			\end{tikzcd}
		\end{equation*}
		is equivalent to the lifting problem
		\begin{equation*}
			\begin{tikzcd}
				J\join B \underset{J\join A}{\amalga} K\join A  \arrow[r] \arrow[d]      &   X  \arrow[d] \\
				K\join B  \arrow[r]    \arrow[ur, dashed]            &   Y\, .
			\end{tikzcd}
		\end{equation*}
		Now the result follows from \cref{lem:left_anodyne_stability}.
	\end{proof}
	\begin{cor}\label{cor:segal_space_slice}
		If $W$ is a marked complete semi-Segal space, and $p\colon J\to W$ a map from a marked semi-simplicial set, then $W_{p/}\to W$ is orthogonal to $J_{JL}$. In particular, $W_{p/}$ is also a complete semi-Segal space.
	\end{cor}
	\begin{proof}  
		Let $A \to B$ be marked inner anodyne. By adjunction, the orthogonal lifting problem
		\begin{equation*}
			\begin{tikzcd}
				A \arrow[r] \arrow[d]      &   W_{p/}  \arrow[d] \\
				B  \arrow[r]  \arrow[ur, dashed]           &   W
			\end{tikzcd}
		\end{equation*}
		is equivalent to the orthogonal lifting problem
		\begin{equation*}
			\begin{tikzcd}
				J\join A \underset{A}{\amalga} B \arrow[r] \arrow[d]      &   W  \arrow[d] \\
				J\join B  \arrow[r] \arrow[ur, dashed]               &   T^{\sharp} \, .
			\end{tikzcd}
		\end{equation*}
		By \cref{lem:left_anodyne_stability}, the left vertical map is marked inner anodyne.
	\end{proof}

	\begin{cor}\label{cor:inner_kan_slice}
		If $X$ is a marked inner Kan space, then so is $X_{p/}$, and map $X_{p/}\to \cT(X)_{up/}$ induced by the localization unit $u\colon id \to \cT$ factors through an equivalence
		\begin{align*}
			\alpha \colon \cT(X_{p/}) \to \cT(X)_{up/} \, ,
		\end{align*}
		which moreover commutes with projection to $\cT(X)$. 
	\end{cor}
	\begin{proof}
		By \cref{cor:left_fibration_stability} the map $u_{*}\colon X_{p/} \to \cT(X)_{up/}$ induced by $u$ is a trivial fibration. The target is a complete semi-Segal space by \cref{cor:segal_space_slice}, and therefore also a marked inner Kan space, so $X_{p/}$ is also a marked inner Kan space. By \cref{cor:T_is_L}, $X_{p/} \to \cT(X)_{up/}$ factors through $u_{X_{p/}}$ by a map $\alpha \colon \cT(X_{p/}) \to \cT(X)_{p/}$. Since $u_{*}$ and $u_{X_{p/}}$ are trivial fibrations, and therefore $\CsSS$-equivalences by \cref{lem:trivial_fibration=>eq}, $\alpha$ is also a $\CsSS$-equivalence by 2-out-of-3. Since both the source and target of $\alpha$ are in $\CsSS$, it is also a levelwise equivalence.
		
		To see that $\alpha$ is compatible with projection to $\cT(X)$, consider the following diagram.
		\begin{equation*}
			\begin{tikzcd}
				X_{p/} \arrow[r, "u"] \arrow[d]      &   \cT(X_{p}/)  \arrow[d]  \arrow[r, "\alpha"] & \cT(X)_{up/} \arrow[d] \\
				X  \arrow[r, "u"]               &   \cT(X)  \arrow[r, "id"]            & \cT(X)
			\end{tikzcd}
		\end{equation*}
		The left square commutes by naturality of $u$, the outer square commutes by naturality of slices. Therefore, the right square commutes by the uniqueness part of the universal property of the $\CsSS$ localization.
	\end{proof}
	By combining \cref{cor:inner_kan_slice} and \cref{cor:R_slice_comparison}, we immediately get the following. 
	\begin{cor}\label{cor:slice_comparison}
		If $X$ is a marked inner Kan space, $J$ a semi-simplicial set, and $p\colon J^{\flat}\to X$ a map, there is a natural equivalence of complete Segal spaces
		\begin{align*}
			\cR_{+}\cT(X_{p/}) \simeq (\cR_{+}\cT(X))^{\overline{up}/} \, .
		\end{align*}
		This equivalence commutes with the forgetful maps to $\cR_{+}\cT(X)$.
	\end{cor}
	We end this section by applying the theory of slices to compute mapping spaces in $\cR_{+}\cT(X)$.
	
	\begin{lem}\label{lem:left_fib_over_point}
		If $X\to T^{\sharp}$ is a marked left fibration, then $X=Y^{\sharp}$ for some semi-Kan space $Y$.
	\end{lem}
	\begin{proof}
		Let $f\colon x\to y$ be any edge in $X$.
		Because $X$ is also an inner Kan space, the points $x$ and $y$ extend to maps $\widetilde{x},\widetilde{y} \colon T^{\sharp} \to X$.
		The image of the unique edge in $T^{\sharp}$ under this map gives us marked edges $sx$ and $sy$ in $X$.
		These should be thought of as choices of unit at $x$ and $y$.
		Now because $X\to T^{\sharp}$ is a left fibration, all horns of the type $\Horn_{0,s}^{2}\to X$ can be filled.
		Doing so for a horn with $f$ and $sx$ produces an edge $g\colon y \to x$, and a 2-simplex witnessing $g\circ f = sx$.
		We again fill a horn with $g$ and $sy$ to produce an edge $f'\colon y\to x$, and a 2-simplex witnessing $f'\circ g = sy$.
		We now have the following ``diagram'' in $X$.
		\begin{equation*}
			\begin{tikzcd}
				x \arrow[r, "f"'] \arrow[rr, bend left, "sx"] & y \arrow[rr, bend right, "sy"'] \arrow[r, "g"] & x \arrow[r, "f' "] & y
			\end{tikzcd}
		\end{equation*}
		The inclusion of the indicated 2-simplices into $\Delta_{s}^{3}$ is inner anodyne, and because $sx$, and $sy$ were marked, the above diagram extends to a $(\Delta^{3}_{s}, S_{2/6}) \to X$.
		Because $X$ is a marked inner Kan space, the diagram further lifts to a map $\Delta_{s}^{3,\sharp}\to X$, which shows that $f,g$, and $f'$ were also marked.
		This shows that all edges in $X$ are marked, and so its underlying semi-simplicial space must be a semi-Kan space.
	\end{proof}
	\begin{lem}\label{lem:T_of_groupoid}
		If $Y$ is a semi-Kan space, then $\cT(Y^{\sharp})$ is the constant semi-simplicial space at $\Vert Y \Vert$ with all edges marked.
	\end{lem}
	\begin{proof}
		For a semi-simplicial map $\Delta^{m}_{s}\to \Delta^{n}_{s}$, we have a commutative square 
		\begin{equation}\label{eq:T_of_groupoid}
			\begin{tikzcd}
				\cT(Y^{\sharp})(\Delta^{n,\sharp}_{s}) \arrow[r] \arrow[d]      &   \cT(Y^{\sharp})(\Delta^{m,\sharp}_{s})  \arrow[d] \\
				\cT(Y^{\sharp})(\Delta^{n,\flat}_{s})  \arrow[r]               &  \cT(Y^{\sharp})(\Delta^{m,\flat}_{s}) \, .
			\end{tikzcd}
		\end{equation}
		By definition, $\cT(Y^{\sharp})$ is marked by the image of
		\begin{align*}
			Y^{\sharp}(\Delta^{1,\sharp}_{s}) \to Y^{\sharp}(\Delta^{1,\flat}_{s}) \to \Vert \widetilde{U}(Y^{\Delta^{1,\flat}_{s}}) \Vert \, ,
		\end{align*}
		where the first map is an equivalence, and the second is essentially surjective. Hence, all edges of $\cT(Y^{\sharp})$ are marked, so the horizontal maps of \cref{eq:T_of_groupoid} are equivalences. Because $Y$ is semi-kan, $Y^{\sharp}$ is marked inner Kan, so $\cT(Y^{\sharp})$ is a complete semi-Segal space. The upper horizontal map of \cref{eq:T_of_groupoid} is determined by restricting along a marked inner anodyne map, and therefore an equivalence. Hence $\cT(Y^{\sharp})$ is a constant semi-simplicial space, and the zeroeth space is by construction $\Vert \widetilde{U}(Y^{\sharp}) \Vert = \Vert Y \Vert$. 
	\end{proof}

	\begin{cor}\label{lem:left_mapping_space}
		Given two points $x,y\colon \Delta_{s}^{0} \to X$ in a marked inner Kan space $X$, the mapping space $\Map_{\cR_{+}\cT(X)}(x,y)$ between the corresponding objects of the complete Segal space $\cR_{+}\cT(X)$ is equivalent to the realization of the semi-Kan space
		\begin{align}\label{eq: }
			\left(\Hom^{L}_{X}(x,\widetilde{y})\right)_n = \left\{ x\right\} \timest_{X(\Delta_{s}^{0})  }X(\Delta_{s}^{0} \join \Delta_{s}^{n,\flat}) \timest_{X(\Delta_{s}^{n,\flat})} \left\{ \widetilde{y} \right\} \, ,
		\end{align}
		where $\widetilde{y}$ is any extension of $y$ to a map $T^{\sharp}\to X$.
	\end{cor}
	\begin{proof}
		Because $X_{x/}\to X$ is a left fibration, we have a pullback square of marked semi-simplicial spaces 
		\begin{equation*}
			\begin{tikzcd}
				\Hom^{L}_{X}(x,y)^{\sharp} \arrow[r] \arrow[d]      &   X_{x/}  \arrow[d] \\
				T^{\sharp}  \arrow[r, "\tilde{y}"]               &   X
			\end{tikzcd}
		\end{equation*}
		which furthermore satisfies the assumptions of \cref{lem:inner_fibration_pullback}. Therefore, applying $\cR_{+}\cT$ to this, and using \cref{cor:slice_comparison,lem:T_of_groupoid}, we get a pullback square
		\begin{equation*}
			\begin{tikzcd}
				\Vert \Hom^{L}_{X}(x,\tilde{y})\Vert \arrow[r] \arrow[d]      &   \cR_{+}\cT(X)^{\overline{ux}/}   \arrow[d] \\
				\Delta^{0}  \arrow[r]               &   \cR_{+}\cT(X)
			\end{tikzcd}
		\end{equation*}
		where the top left corner is a constant simplicial space. Evaluating this at $\Delta^{0}$ and using \cref{rem:mapping_space_from_slice} then gives the result.
	\end{proof}
	
	\section{Comparison of quasi-unitality conditions} \label{sec:alternative_condition_for_quasi_unitality}
	In this subsection, we compare three conditions for quasi-unitality. These are
	\begin{enumerate}
		\item Being a marked inner Kan space as in \cref{def:generating_marked_cofibrations}, generalizing \cite{henry2018weak}.
		\item Each vertex admitting an idempotent equivalence as in \cite{steimle2018degeneracies} and \cite{tanaka2018functors}.
		\item Existence of terminal and initial degeneracies $s_0, s_n \colon X_n \to X_{n+1}$ as in \cite{AB}.
	\end{enumerate}
	We begin by comparing the first two.
	\begin{definition}\label{def:unmarked_fibrations}
		Consider the following collections of maps in $ss\cS$.
		\begin{enumerate}
			\item  $J_{L} = \left\{ \Horn^{n}_{i,s}\to \Delta^{n}_{s}, 0\leq i <n \right\}$
			\item  $J_{R} = \left\{ \Horn^{n}_{i,s}\to \Delta^{n}_{s}, 0 < i \leq n \right\}$
			\item  $J_{I} = \left\{ \Horn^{n}_{i,s}\to \Delta^{n}_{s}, 0< i <n \right\}$
		\end{enumerate}
		We call $J_{L}$-, $J_{R}$-, and $J_{I}$-fibrations \emph{left, right,} and \emph{inner fibrations}, respectively. We call the corresponding classes of cofibrations \emph{left, right,} and \emph{inner} anodyne, respectively.
	\end{definition} 
	\begin{rem}\label{rem:unmarked_fibration_stability}
		Because the proof of \cite{HTT} Lemma 2.1.2.3 goes through in the unmarked setting, the obvious analogues of \cref{lem:left_anodyne_stability,cor:left_fibration_stability} hold in $ss\cS$. 
	\end{rem}
	\begin{definition}\label{def:equivalence_in_inner_Kan}
		An edge $f$ in an inner Kan space $X$ is called an \emph{equivalence} if we can solve the lifting problems
		\begin{equation*}
			\begin{tikzcd}
				\Horn^{n}_{0,s} \arrow[r, "F"] \arrow[d, hookrightarrow] & X \\
				\Delta^{n}_{s} \arrow[ur, dotted]
			\end{tikzcd}
			\quad
			\begin{tikzcd}
				\Horn^{n}_{n,s} \arrow[r, "G"] \arrow[d, hookrightarrow]& X \\
				\Delta^{n}_{s} \arrow[ur, dotted]
			\end{tikzcd}
		\end{equation*}
		up to homotopy for $2\leq n$, whenever the first edge of $F$ (resp. the last edge of $G$) is homotopic to $f$. 
	\end{definition}
	\begin{rem}\label{rem:equivalences_is_subspace}
		Note that by definition, the inclusion of the subspace $X_{1}^{eq} \subset X_{1}$ of equivalences is a monomorphism.
	\end{rem}
	\begin{rem}\label{rem:marked=>eq}
		Because the marked horn inclusions $(\Horn^{n}_{0,s}, \left\{ 0,1 \right\} ) \to (\Delta^{n}_{s}, \left\{ 0,1\right\})$ and $(\Horn^{n}_{n,s}, \left\{ n-1,n \right\} ) \to (\Delta^{n}_{s}, \left\{ n-1,n \right\})$ are in $J_{JL}$, any marked edge in an inner Kan space $X$ is an equivalence in $U(X)$. 
	\end{rem}
	\begin{definition}\label{def:idempotent}
		An edge $f$ in an inner Kan space $X$ is called \emph{idempotent} if
		\begin{enumerate}
			\item $d_0 f \sim d_1 f \in X_0$.
			\item There exists a 2-cell $H\in X_2$ with $d_i H \sim f$ for $i=0,1,2$.
		\end{enumerate}
	\end{definition}
	\begin{definition}\label{def:quasi_unital}
		We say that an inner Kan space $X$ is \emph{quasi-unital} if for every $x\in X_0$, there exists an idempotent equivalence $e$ at $x$. We say that a map $F\colon X\to Y$ between semi-simplicial spaces is quasi-unital if it takes idempotent equivalences to idempotent equivalences. 
	\end{definition}
	\begin{rem}\label{rem:quasi_unital_map}
		Note that the image of an idempotent is always an idempotent, so it suffices to assume that each $F(e)$ is an equivalence. In fact, it follows from the proof of \cref{cor:qu_iff_preserves_eqs} that for $X$ and $Y$ quasi-unital, it is sufficient to assume that $F\colon X\to Y$ preserves one idempotent equivalence at each vertex $x\in X$.
	\end{rem}
	In \cite{steimle2018degeneracies} it is shown that any levelwise discrete, quasi-unital inner Kan space is in the image of the forgetful functor $s\Set \to ss\Set$. Moreover, if we make a choice of idempotent equivalence $e_x$ for each vertex $x\in X_0$, we can take the simplicial structure on $X$ to have $s_{0}x = e_{x}$. Then, by \cite{tanaka2018functors}, a quasi-unital map between levelwise discrete quasi-unital inner Kan spaces lifts, up to unique homotopy, to a simplicial map.
	\begin{definition}\label{def:IK_qu}
		Let $\IK_{qu} \subset ss\cS$ denote the (non-full) subcategory spanned by quasi-unital inner Kan spaces and quasi-unital maps between them. 
	\end{definition}
	\begin{definition}\label{def:natural}
		For a semi-simplicial space $X$, we let $X^{\natural}$ denote the marked semi-simplicial space obtained by marking all equivalences (in the sense of \cref{def:equivalence_in_inner_Kan}) in $X$. 
	\end{definition}
	We will show that the forgetful functor and $(-)^{\natural}$ induce inverse equivalences $\IK_{+} \simeq \IK_{qu}$. We begin with some preliminary lemmas.
	\begin{lem}\label{lem:eq_fibration_condition}
		An edge $f\colon x\to y$ in an inner Kan space $X$ is an equivalence if and only if $X_{f/} \to X_{x/}$ and $X_{/f} \to X_{/y}$ are trivial fibrations.
	\end{lem}
	\begin{proof}
		A lifting problem of the form
		\begin{equation*}
			\begin{tikzcd}
				\del \Delta^{n}_{s} \arrow[r] \arrow[d]      &   X_{f/}  \arrow[d] \\
				\Delta^{n}_{s}  \arrow[r]  \arrow[ur,dashed]             &   X_{x/}
			\end{tikzcd}
		\end{equation*}
		for $n\geq 0$ is equivalent by adjunction to one of the form
		\begin{equation*}
			\begin{tikzcd}
				\Delta^{1}_{s}\join \del \Delta^{n}_{s} \underset{\Delta^{0}_{s}\join \del \Delta^{n}_{s}}{\amalga} \Delta^{0}_{s} \join \Delta^{n}_{s}  \arrow[r] \arrow[d]      &   X  \arrow[d] \\
				\Delta^{1}_{s}\join \Delta^{n}_{s}  \arrow[r] \arrow[ur, dashed]              &   T
			\end{tikzcd}
		\end{equation*}
		which is precisely a filling of a $\Horn_{0,s}^{n+2}$ with first edge $f$. A dual argument works for $X_{/f} \to X_{/y}$.
	\end{proof}
	\begin{cor}\label{cor:trivial_fibration_is_conservative}
		Let $p\colon X\to Y$ be a trivial fibration between semi-simplicial spaces. An edge $f\colon x\to y$ in $X$ is an equivalence if and only if $p(f)$ is an equivalence in $Y$.
	\end{cor}
	\begin{proof}
		Consider the following diagram.
		\begin{equation*}
			\begin{tikzcd}
				X_{f/} \arrow[r, "a"] &  X_{x/}\timest_{Y_{p(x)/}}  Y_{p(f)/} \arrow[r] \arrow[d]      &   X_{x/}  \arrow[d, "b"] \\
				& Y_{p(f)/}  \arrow[r]               &   Y_{p(x)/}
			\end{tikzcd}
		\end{equation*}
		By \cref{cor:left_fibration_stability}, the maps $a$ and $b$ are trivial fibrations. By pullback stability of trivial fibrations and \cref{lem:trivial_fib_cancellation}, $Y_{p(f)/}\to Y_{p(x)/}$ is a trivial fibration if and only if $X_{f/}\to X_{x/}$ is a trivial fibration. A dual argument shows that $Y_{/p(f)} \to Y_{/p(y)}$ is a trivial fibration if and only if $X_{/f} \to X_{/y}$ is. Now by \cref{lem:eq_fibration_condition}, $p(f)$ is an equivalence if and only if $f$ is. 
	\end{proof}
	\begin{cor}\label{cor:trivial_fibration_preserves_quasi_unitality}
		If $p\colon X\to Y$ is a trivial fibration, and $Y$ is an inner Kan space, then $X$ is also an inner Kan space, which is moreover quasi-unital if and only if $Y$ is.
	\end{cor}
	\begin{proof}
		The inner Kan space property is obvious because trivial fibrations are also inner fibrations. Now assume $Y$ is quasi-unital. For any vertex $x\in X_0$, take a 2-cell $H \in Y_2$ witnessing an idempotent equivalence $e\colon p(x) \to p(x)$. Use the trivial Kan fibration property, first for $\del \Delta^{1}_{s} \to \Delta^{1}_{s}$ to lift $e$ to an edge $e':x\to x$ in $X$, and then for $\del \Delta^{2}_{s} \to \Delta^{2}_{s}$ to lift $H$ to a 2-cell $H'$ witnessing $e'$ as idempotent. Now because $p(e')$ is an equivalence, so is $e$ by \cref{cor:trivial_fibration_is_conservative}.
		
		Now assume $X$ is quasi-unital. For any vertex $y\in Y_0$, lift to a vertex $x\in X_0$ such that $p(x)=y$. Let $H$ be a 2-simplex witnessing an idempotent equivalence at $x$. Then by \cref{cor:trivial_fibration_is_conservative}, $p(H)$ witnesses an idempotent equivalence at $y$. 
	\end{proof}
	\begin{lem}\label{lem:equivalence_in_quasi-category}
		Let $X \in s\Set$ be a quasi-category. Then an edge $f$ in $X$ is invertible if and only if the corresponding edge in $\cF(X)$ is an equivalence in the sense of \cref{def:equivalence_in_inner_Kan}.
	\end{lem}
	\begin{proof}
		Because join commutes with left Kan extension, we get by uniqueness of right adjoints that for any semi-simplicial map $p\colon J\to \cF(X)$ the slice $\cF(X)_{p/}$ agrees with $\cF(X_{\tilde{p}/})$ where $\tilde{p}\colon \cL(J) \to X$ is the adjunct of $p$. Under unstraightening, the natural transformation $f^{*}\colon \Map_{X}(y,-) \to \Map_{X}(x, -)$ between copresheaves is represented by the span of left fibrations
		\begin{align*}
			X_{y/} \xleftarrow{\sim} X_{f/} \rightarrow X_{x/} \, .
		\end{align*}
		Hence, it follows from the Yoneda lemma that $f$ is an equivalence in $X$ if and only if $X_{f/} \to X_{x/}$ is an equivalence. A left fibration between quas-categories is an equivalence if and only if it is a trivial fibration. Trivial fibrations are detected by $\cF$, and by the above $\cF(X_{f/}) \to \cF(X_{x/})$ is a trivial fibration if and only if $\cF(X)_{f/} \to \cF(X)_{x/}$ is. A dual argument shows that $f$ is an equivalence in $X$ if and only if $\cF(X)_{/f} \to \cF(X)_{/y}$ is a trivial fibration. It then follows from \cref{lem:eq_fibration_condition} that $f$ is an equivalence in $X$ if and only if the adjunct edge in $\cF(X)$ is an equivalence.
	\end{proof}
	\begin{cor}\label{cor:2-out-of-6_for_eqs}
		The collection of equivalences in a quasi-unital inner Kan space satisfies the 2-out-of-6 property.
	\end{cor}
	\begin{proof}
		Let $X$ be a quasi-unital inner Kan space, and $H\colon \Delta^{3}_{s} \to X$ a 3-cell with edges
		\begin{equation*}
			\begin{tikzcd}
				x \arrow[r, "f"'] \arrow[rr, bend left, "e_{x}"] & y \arrow[r, "g"] \arrow[rr, bend right, "e_{y}"'] & \tilde{x} \arrow[r, "\tilde{f}"] &  \tilde{y} \, ,
			\end{tikzcd}
		\end{equation*}
		where $e_x$ and $e_{y}$ are equivalences. Pick a trivial Kan fibration $X'\to X$ from a semi-simplicial set, and lift $H$ to a map $H'\colon \Delta_{s}^{3} \to X'$. By \cref{cor:trivial_fibration_is_conservative}, the edges $e'_{x}$ and $e'_{y}$ of $H'$ that lift $e_{x}$ and $e_{y}$, respectively, are equivalences. Hence by \cref{cor:trivial_fibration_preserves_quasi_unitality} $X'$ is a quasi-unital quasi-category, so by \cite{steimle2018degeneracies}, $X'=\cF(\overline{X})$ for a quasi-category $\overline{X}$. Now consider the adjunct $H''$ of $H'$. 
		By naturality of adjuncts, $H''$ has edges that are adjuncts of the edges in $H'$. In particular, by \cref{lem:equivalence_in_quasi-category}, the edges $e''_{x}$ and $e''_{y}$, which are adjuncts of $e'_{x}$ and $e'_{y}$, are equivalences in $\overline{X}$. Now equivalences in a quasi-category have the 2-out-of-6 property, so all the edges in $H''$ are equivalences. Then by \cref{cor:trivial_fibration_is_conservative,lem:equivalence_in_quasi-category}, so are all the edges in $H$.
	\end{proof}
	\begin{lem}\label{lem:eqs_iff_invertible}
		An edge $f\colon x\to y$ in a quasi-unital inner Kan space $X$ is an equivalence if and only if for any choice of idempotent equivalences $sx\colon x\to x$ and $sy\colon y\to y$ there exists an edge $g\colon y \to x$ and an $H\in X_{3}$ with edges
		\begin{equation}\label{eq:inverse_witness}
			\begin{tikzcd}
				x \arrow[r, "f"'] \arrow[rr, bend left, "sx"] & y \arrow[r, "g"] \arrow[rr, bend right, "sy"'] & x \arrow[r, "f"] & y \, .
			\end{tikzcd}
		\end{equation}
	\end{lem}
	\begin{proof}
		The only if part follows immediately from \cref{cor:2-out-of-6_for_eqs}. Assume that $f\colon x\to y$ is an equivalence, and $sx$ and $sy$ are idempotent equivalences at $x$ and $y$, respectively. Because $sx$ is an equivalence, we can find an edge $\tilde{f}$ and a 2-simplex $\tilde{s}_{0}f$ witnessing $\tilde{f}\circ sx = f$. Now fill a horn $\Horn^{3}_{2,s} \to X$ constructed by gluing two copies of $\tilde{s}_0f$ along $\tilde{f}$, and a 2-simplex $ssx$ witnessing the idempotence of $sx$. The $d_2$ face $s_0f$ of the resulting simplex witnesses $f\circ sx = f$. Similarly, we can find a 2-simplex $s_1 f$ witnessing $sy\circ f = f$. Now because $f$ is an equivalence, we can fill a $\Horn^{2}_{0,s}$ to find an edge $g$ and a 2-simplex $d_3H$ witnessing $g\circ f = sx$. Now construct a $\Horn^{3}_{0,s} \to X$ by gluing the 2-simplices $d_3 H, s_0f$ and $s_1 f$ along their common edges. Because $f$ is an equivalence, we can fill this horn to obtain the desired simplex $H$. 
	\end{proof}
	\begin{rem}
		We say that a $3$-simplex $H$ with edges as in \cref{eq:inverse_witness} witnesses an inverse to $f$ with respect to $sx$ and $sy$.
	\end{rem}
	\begin{cor}\label{cor:qu_iff_preserves_eqs}
		A map $F\colon X\to Y$ between quasi-unital inner Kan spaces preserves equivalences if and only if it preserves idempotent equivalences. 
	\end{cor}
	\begin{proof}
		Any semi-simplicial map preserves idempotents, and so the only if part follows immediately. Assume that $F$ preserves idempotent equivalences, and let $f\colon x\to y$ be an equivalence in $X$. Choose idempotent equivalences $sx$ and $sy$ at $x$ and $y$, respectively. By \cref{lem:eqs_iff_invertible} we may choose a 3-simplex $H$ in $X$ witnessing an inverse of $f$ with respect to $sx$ and $sy$. Then $F(H)$ is a 3-simplex which, because $F$ preserves idempotent equivalences, witnesses an inverse of $F(f)$ with respect to $F(sx)$ and $F(sy)$. In particular, $F(f)$ is an equivalence by \cref{lem:eqs_iff_invertible}.
	\end{proof}
	The previous corollary shows that the assignment $X\mapsto X^{\natural}$ defines a functor $\IK_{qu} \to ss\cS_{+}$. We now show that this functor factors through the full subcategory of marked inner Kan spaces.
	\begin{cor}\label{lem:qu_iff_enough_eqs}
		An inner Kan space $X$ is quasi-unital if and only if $X^{\natural}$ is a marked inner Kan space.
	\end{cor}
	\begin{proof}
		Assume first that $X$ is quasi-unital. Then $X^{\natural} \to T^{\sharp}$ has the right lifting property with respect to:
		\begin{enumerate}
			\item $d_{i}\colon \Delta^{0}_{s} \to \Delta^{1,\sharp}_{s}$ because each vertex admits an idempotent equivalence.
			\item Unmarked inner horn inclusions because $X$ is an inner Kan space.
			\item The marked outer horn inclusions by definition of equivalences.
			\item $S_{2/6}$ by \cref{cor:2-out-of-6_for_eqs}.
		\end{enumerate}
		For the converse, it follows from \cref{cor:tildeU_fibrations} that the subspace
		\begin{align*}
			X^{\simeq}\coloneq \widetilde{U}(X^{\natural}) \subset U(X^{\natural}) = X
		\end{align*}
		of simplices with all edges equivalences is an inner Kan space. The map $\Delta^{0}\to T$ is anodyne by \cref{lem:unit_is_anodyne}, so every vertex $x\colon \Delta^{0}_{s} \to X$ extends to a map $T \to X^{\simeq}$, whose value at $\Delta^{2}_{s}$ exhibits an idempotent equivalence at $x$.
	\end{proof}

	\begin{prop}\label{prop:IK_+=IK_qu}
		The functors $U$ and $(-)^{\natural}$ restrict to inverse equivalences $\IK_{+} \simeq \IK_{qu}$. 
	\end{prop}
	\begin{proof}
		The key fact left to show is that if $X$ is a marked inner Kan space, then an edge $f\colon x\to y$ in $U(X)$ is an equivalence if and only if it lifts to a marked edge in $X$, in other words, that $X \simeq U(X)^{\natural}$. The if direction follows immedatly from $X$ being a marked inner Kan space, so assume that $f$ is an equivalence. By \cref{cor:tildeU_fibrations}, the subspace $\widetilde{U}(X) \subset U(X)$ of marked edges is a semi-Kan space, so using \cref{lem:unit_is_anodyne} to extend the vertices $x$ and $y$ to maps $T \to \widetilde{U}(X)$, we get marked idempotents $sx$ and $sy$ at $x$ and $y$, respectively. Because marked edges are also equivalences, we can find a $3$-simplex $H$ witnessing an inverse of $f$ with respect to $sx$ and $sy$. Now, since the collection of marked edges in $X$ has the 2-out-of-6 property, $f$ must be marked in $X$. 
		
		Combining the above with \cref{cor:qu_iff_preserves_eqs,lem:qu_iff_enough_eqs} shows that the restriction of the forgetful functor factors as $U\colon \IK_{+} \to \IK_{qu}$. Also by \cref{cor:qu_iff_preserves_eqs,lem:qu_iff_enough_eqs}, the construction $X^{\natural}$ determines a functor $(-)^{\natural}\colon \IK_{qu} \to \IK_{+}$. We clearly have a natural equivalence $id_{\IK_{qu}} \simeq U( (-)^{\natural})$, and as we argued above, $id_{\IK_{+}} \simeq U(-)^{\natural}$.
	\end{proof}
	\begin{cor}\label{cor:natural_preserves_exponential}
		For a quasi-unital inner Kan space $X$ and a semi-simplicial set $J$, the exponential $X^{J}$ is also a quasi-unital inner Kan space, and there is a natural equivalence $(X^{J})^{\natural} \simeq (X^{\natural})^{J^{\flat}}$.
	\end{cor}
	\begin{proof}
		By \cref{prop:IK_+=IK_qu,cor:marked_fibration_stability},  $(X^{\natural})^{J^{\flat}}$ is a marked inner Kan space. By \cref{prop:IK_+=IK_qu}, it therefore suffices to show that the underlying semi-simplicial space of $(X^{\natural})^{J^{\flat}}$ is $X^{J}$. The left adjoint $(-)^{\flat}$ is strongly monoidal with respect to $\otimes$, giving a commutative square
		\begin{equation*}
			\begin{tikzcd}
				ss\cS \arrow[r, "(-)^{\flat}"] \arrow[d, "-\otimes J"']      &   ss\cS_{+}  \arrow[d, "-\otimes J^{\flat}"] \\
				ss\cS  \arrow[r, "(-)^{\flat}"']               &   ss\cS_{+} \, .
			\end{tikzcd}
		\end{equation*}
		By uniqueness of right adjoints, we therefore get a commutative square
		\begin{equation*}
			\begin{tikzcd}
				ss\cS     & \arrow[l, "U"']   ss\cS_{+} \\
				ss\cS \arrow[u, "(-)^{J}"],  &\arrow[l, "U"]   \arrow[u, "(-)^{J^{\flat}}"']               ss\cS_{+} \, ,
			\end{tikzcd}
		\end{equation*}
		which when applied at $X^{\natural}$ gives 
		\begin{align*}
			U( (X^{\natural})^{J^{\flat}} ) \simeq U(X^{\natural})^{J} \simeq X^{J} \, . & \qedhere
		\end{align*}
	\end{proof}
	\begin{cor}\label{cor:natural_preserves_slices}
		For a quasi-unital inner Kan space $X$, a semi-simplicial set $J$ and a map $p\colon J\to X$, let $\overline{p}$ denote the unique map  $J^{\flat} \to X^{\natural}$ lifting $p$. Then $X_{p/}$ is also a quasi-unital inner Kan space, and there is a natural equivalence $(X_{p/})^{\natural} \simeq (X^{\natural})_{\overline{p}/}$.
	\end{cor}
	\begin{proof}
		By \cref{prop:IK_+=IK_qu}, $X^{\natural}$ is a marked inner Kan space, and by \cref{cor:left_fibration_stability}, $(X^{\natural})_{\overline{p}/} \to X^{\natural}$ is a marked left fibration, so $(X^{\natural})_{\overline{p}/}$ is also a marked inner Kan space. By \cref{prop:IK_+=IK_qu}, it therefore suffices to show that the underlying semi-simplicial space of $(X^{\natural})_{\overline{p}/}$ is $X_{p/}$. The left adjoint $(-)^{\flat}$ is strongly monoidal with respect to $\join$, giving a commutative square
		\begin{equation*}
			\begin{tikzcd}
				ss\cS \arrow[r, "(-)^{\flat}"] \arrow[d, "-\join J"']      &   ss\cS_{+} \arrow[d, "-\join J^{\flat}"] \\
				ss\cS_{J/}  \arrow[r, "(-)^{\flat}"']               &   (ss\cS_{+})_{J^{\flat}/} \, .
			\end{tikzcd}
		\end{equation*}
		By uniqueness of right adjoints, we therefore get a commutative square
		\begin{equation*}
			\begin{tikzcd}
				ss\cS     & \arrow[l, "U"']   ss\cS_{+} \\
				ss\cS_{J/} \arrow[u],  &\arrow[l, "U"]   \arrow[u]               (ss\cS_{+})_{J^{\flat}/} \, ,
			\end{tikzcd}
		\end{equation*}
		where the vertical maps compute slices. When applied at $\overline{p}\colon J^{\flat} \to X^{\natural}$, this gives the desired equivalence
		\begin{align*}
			U( (X^{\natural})_{\overline{p}/} ) \simeq U(X^{\natural})_{p/} \simeq X_{p/} \, . & \qedhere
		\end{align*}
	\end{proof}
	The last ingredient we need for the proof of \cref{thm:IK+_model} is that $(-)^{\natural}$ preserves (trivial) fibrations. We start with a preliminary lemma.
	\begin{lem}\label{lem:qu_triv_fib_2/3}
		Let $X\xrightarrow{f} Y \xrightarrow{g} Z$ be inner fibrations between quasi-unital inner Kan spaces. Then if any two of $f,g$ and $g\circ f$ are trivial fibrations, then so is the third.
	\end{lem}
	\begin{proof}
		Fibrations are closed under composition, and the case where $f$ and $g\circ f$ are trivial follows from \cref{lem:trivial_fib_cancellation}. Successively use \cref{lem:simplicial_set_replacement,cor:trivial_fibration_preserves_quasi_unitality} to replace $X, Y$ and $Z$ up to trivial fibration by levelwise discrete, quasi-unital inner Kan spaces $\overline{X}, \overline{Y}$ and $\overline{Z}$. Then, successively use the main result of \cite{steimle2018degeneracies} to pick compatible simplicial structures, producing a diagram
		\begin{equation*}
			\begin{tikzcd}
				\cF(X') \arrow[r]  \arrow[d, "\cF(f')"']      &   X  \arrow[d,"f"] \\
				\cF(Y')  \arrow[r] \arrow[d, "\cF(g')"']      &   Y \arrow[d,"g"]  \\
				\cF(Z')  \arrow[r]                &   Z
			\end{tikzcd}
		\end{equation*}
		where $X', Y'$ and $Z'$ are quasi-categories, the horizontal maps as well as $g'$ and $g'\circ f'$ are trivial fibrations, and $f'$ is an inner fibration.
		By 2-out-of-3 for weak equivalences, $f'$ is also a weak equivalence in the quasi-category model structure on $s\Set$. In a model category, a morphism is a trivial fibration if and only if it is a weak equivalence and a fibration, so $f'$ is a trivial fibration, and therefore so is $\cF(f')$. Now apply \cref{lem:trivial_fib_cancellation} to the top square to see that $f$ is a trivial fibration.
	\end{proof}
	
	\begin{cor}\label{cor:natural_preserves_trivial_fibrations}
		The functor $(-)^{\natural}\colon \IK_{qu} \to \IK_{+}$ preserves trivial fibrations.
	\end{cor}
	\begin{proof}
		Let $X\to Y$ be a trivial fibration. Then $X^{\natural} \to Y^{\natural}$ has the right lifting property with respect to the unmarked boundary inclusions by adjunction. Lifting against $\Delta^{1,\flat}_{s} \to \Delta^{1, \sharp}_{s}$ follows immedeatly from \cref{cor:trivial_fibration_is_conservative}.
	\end{proof}
	
	\begin{cor}\label{cor:natural_preserves_fibrations}
		The functor $(-)^{\natural}\colon \IK_{qu} \to \IK_{+}$ takes inner fibrations to marked inner fibrations.
	\end{cor}
	\begin{proof}
		Let $X\to Y$ be a quasi-unital inner fibration between quasi-unital inner Kan spaces. Then $X^{\natural} \to Y^{\natural}$ has the right lifting property against the unmarked inner horns by adjunction, and against $S_{2/6}$ by \cref{cor:2-out-of-6_for_eqs}. Now the lifting problem
		\begin{equation*}
			\begin{tikzcd}
				(\Horn^{n}_{0,s}, \left\{ 0,1 \right\})  \arrow[r, "H"] \arrow[d]      &   X^{\natural}  \arrow[d] \\
				(\Delta^{n}, \left\{ 0,1 \right\})  \arrow[r]  \arrow[ur, dashed]            &   Y^{\natural}
			\end{tikzcd}
		\end{equation*}
		is equivalent by adjunction to
		\begin{equation*}
			\begin{tikzcd}
				\del \Delta^{n-2} \arrow[r] \arrow[d]      &   X_{f/} \arrow[d, "a"] \\
				\Delta^{n-2}  \arrow[r]  \arrow[ur, dashed]             &   Y_{p(f)/} \timest_{Y_{p(x)/}} X_{x/} \, ,
			\end{tikzcd}
		\end{equation*}
		where $f\colon x\to y$ is the edge $H(\left\{ 0,1\right\})$. It therefore suffices to show that the map $a$ is a trivial fibration. By construction of $X^{\natural}$ and $Y^{\natural}$, the edges $f$ and $p(f)$ are equivalences in $X$ and $Y$, respectively, so we have a diagram
		\begin{equation*}
			\begin{tikzcd}
				X_{f/} \arrow[r, "a"] \arrow[dr, "c"']  &
				Y_{p(f)/} \timest_{Y_{p(x)/}} X_{x/} \arrow[r]  \arrow[d, "b'"] & Y_{p(f)/} \arrow[d, "b"] \\
				& X_{x/} \arrow[r] & Y_{p(x)/} \, ,
			\end{tikzcd}
		\end{equation*}
		where the bottom square is a pullback, and the maps $b$ and $c$ are trivial fibrations. By pullback stability, $b'$ is also a trivial fibration. By \cref{cor:trivial_fibration_preserves_quasi_unitality,cor:natural_preserves_slices}, all the objects in this diagram are quasi-unital inner Kan spaces, so we can apply \cref{lem:qu_triv_fib_2/3} to the left triangle to conclude that $a$ is also a trivial fibration.
	\end{proof}

	\begin{proof}[Proof of \cref{thm:IK+_model}]
		The localization statement follows by \cref{cor:T_characterization,cor:T_is_L}, and from the equivalence $\IK_{+} \simeq \IK_{qu}$ preserving trivial fibrations. The preservation of fibrant limits follows from \cref{lem:inner_fibration_pullback,cor:inverse_limit_marked,cor:natural_preserves_fibrations}. Preservation of exponentials follows from \cref{cor:natural_preserves_exponential,cor:RT(X)^J}. Preservation of slices follows from \cref{cor:natural_preserves_slices,cor:slice_comparison}.
	\end{proof}

	We now consider a sufficient condition for quasi-unitality inspired by \cite{AB}. We begin by formulating the precise coherence required for outer degeneracies of semi-simplicial spaces.
	\begin{definition}\label{def:outer_degeneracies}
		We say that a semi-simplicial space $X$ \emph{admits initial and terminal degeneracies} if there exist commutative diagrams of augmented semi-simplicial spaces
		\begin{equation*}
			\begin{tikzcd}
				X_{\Delta_{s}^{0}/} \arrow[r, "s_0"] \arrow[dr, "id"']      &   X_{\Delta_{s}^{1}/}  \arrow[d, "d_{i}"] \\
				&   X_{\Delta_{s}^{0}/}
			\end{tikzcd}
			\qquad
			\begin{tikzcd}
				X_{/\Delta_{s}^{0}} \arrow[r, "s_{\omega}"] \arrow[dr, "id"']      &   X_{/\Delta_{s}^{1}}  \arrow[d, "d_{i}"] \\
				&   X_{/\Delta_{s}^{0}} \, ,
			\end{tikzcd}
		\end{equation*}
		where the maps $\Delta$ are restriction along the unique map $\del \Delta^{1}_{s} \to \Delta^{0}_{s}$,
		such that for every vertex $x\in X_{0}$, the edges $s_{0}x$ and $s_{\omega}x$ belong to the same component of $X_{1}$.
	\end{definition}
	Under the inclusion $ss\cS \subset ss\cS_{a}$, we can consider trivial, inner, left and right fibrations in $ss\cS_{a}$. Because the inclusion of $ss\cS$ preserves pushouts and retracts, the corresponding collections of anodynes are preserved. Note that a map $f\colon X\to Y$ in $ss\cS_{a}$ is a fibration precisely if for every $x\in X_{-1}$, the map $X_{x} \to Y_{f(x)}$ is a fibration in $ss\cS$. We therefore have the following immediate corollary of \cref{lem:free_slice,rem:unmarked_fibration_stability}.
	\begin{cor}\label{cor:free_slice_fibration}
		Let $X\to Y$ be an inner fibration in $ss\cS$, and $j\colon J\to K$ a map of semi-simplicial sets. Then
		\begin{align*}
			X_{K/} \to X_{J/} \timest_{Y_{J/}} Y_{K/}
		\end{align*}
		is a left fibration, and if $j$ is right anodyne, a trivial fibration. 
	\end{cor}
	We now assume that $X$ admits initial degeneracy $s = s_0$. By associativity of $\join$, we have maps
	\begin{align*}
		s_{J/} \colon X_{\Delta^{0}_{s} \join J /} \simeq (X_{\Delta^{0}_{s} /})_{J/} \to (X_{\Delta^{1}_{s} /})_{J/} \simeq X_{\Delta^{1}_{s}\join J/}
	\end{align*}
	which are natural in $J$. In particular, we get maps $X_{\Delta^{n}_{s}/} \to X_{\Delta^{n+1}_{s}}$, each a section of the restriction maps $d_0$ and $d_{1}$. Consider the map
	\begin{align*}
		S\colon X_{\Delta^{1}_{s}/} \simeq X_{\Delta^{1}_{s}/} \timest_{X_{\Delta^{0}_{s}/}} X_{\Delta^{0}_{s}/} \xrightarrow{id\timest s} X_{\Delta^{1}_{s}/} \timest_{X_{\Delta^{0}_{s}/}} X_{\Delta^{1}_{s}/} \simeq X_{\Horn^{2}_{0,s}/} \, ,
	\end{align*}
	where the pullback is formed over the span $X_{\Delta^{1}_{s}/} \xrightarrow{d_0} X_{\Delta^{0}_{s}/} \xleftarrow{d_0} X_{\Delta^{1}_{s}/}$. 
	\begin{lem}\label{lem:S*X_over_delta*2}
		If $X$ is an inner Kan space which admits initial degeneracies, then the left vertical map of the pullback
		\begin{equation*}
			\begin{tikzcd}
				S^{*} X_{\Delta^{2}_{s}/} \arrow[r] \arrow[d]      &   X_{\Delta^{2}_{s}/}  \arrow[d] \\
				X_{\Delta^{1}_{s}/}  \arrow[r, "S"]               &   X_{\Horn^{2}_{0,s}/}
			\end{tikzcd}
		\end{equation*}
		has the right lifting property with respect to the unique map $\underline{\ast} \to J$ in $ss\cS_{a}$ for any semi-simplicial set $J$. 
	\end{lem}
	\begin{proof}
		The inclusion $\Horn^{2}_{0,s}\join \Delta^{0}_{s} \simeq \Delta^{0}_{s} \join \Horn^{2}_{2,s} \to \Delta^{0}_{s} \join \Delta^{2}_{s}\simeq \Delta^{3}_{s}$ is right anodyne, so $X_{\Delta^{3}_{s}/} \to X_{\Delta^{0}_{s}\join \Horn^{2}_{2,s}/}$ is a trivial fibration by \cref{cor:free_slice_fibration}. Then, thanks to the factorization
		\begin{equation*}
			\begin{tikzcd}
				X_{\Delta^{3}_{s}/} \arrow[r, "d_{3}"] \arrow[d]      &   X_{\Delta^{2}_{s}/}  \arrow[d] \\
				X_{\Horn^{2}_{0,s}\join \Delta^{0}_{s}/}  \arrow[r, "d_{3}"]               &   X_{\Horn^{2}_{0,s}/} \, ,
			\end{tikzcd}
		\end{equation*}
		it suffices to show that $S^{*}d_{3}\colon S^{*}X_{\Horn^{2}_{0,s}\join \Delta^{0}_{s}/} \to X_{\Delta^{1}_{s}/}$ has the right lifting property with respect to $\underline{\ast} \to J$. Consider the map
		\begin{align*}
			T\colon X_{\Delta^{2}_{s}/} \simeq X_{\Delta^{2}_{s}/} \timest_{X_{\Delta^{1}_{s}/}} X_{\Delta^{1}_{s}/} \xrightarrow{id \timest s} X_{\Delta^{2}_{s}/} \timest_{X_{\Delta^{1}_{s}}/} X_{\Delta^{2}_{s}/} \simeq X_{\Horn^{2}_{0,s}\join \Delta^{0}_{s}/} \, ,
		\end{align*}
		which by naturality of $s$ fits in the following square.
		\begin{equation*}
			\begin{tikzcd}
				X_{\Delta^{2}_{s}/} \arrow[r, "T"] \arrow[d, "d_{2}"]      &   X_{\Horn^{2}_{0,s}\join \Delta^{0}_{s}/}  \arrow[d, "d_{3}"] \\
				X_{\Delta^{1}_{s}/}  \arrow[r, "S"]               &   X_{\Horn^{2}_{0,s}/}
			\end{tikzcd}
		\end{equation*}
		Because the right vertical is a trivial fibration, it suffices to show that the left vertical $d_{2}$ has the right lifting property with respect to $\underline{\ast} \to J$. This map factors as
		\begin{align*}
			X_{\Delta^{2}_{s}/} \to X_{\Horn^{2}_{1,s}/} \to X_{\Delta^{1}_{s}/} \, ,
		\end{align*}
		where the first factor is a trivial fibration by \cref{cor:free_slice_fibration}, and the second factor admits a section constructed in a similar fashion to $S$.
	\end{proof}
	\begin{lem}\label{lem:sx_is_left_eq}
		If an inner Kan space $X$ admits initial degeneracies $s$, then for any vertex $x\in X_{0}$, the map $X_{sx/} \xrightarrow{d_1} X_{x/}$ is a trivial fibration.
	\end{lem}
	\begin{proof}
		Let $J\to K$ be a levelwise injection between semi-simplicial sets. By adjunction, the lifting problem 
		\begin{equation*}
			\begin{tikzcd}
				J \arrow[r] \arrow[d]      &   X_{sx/}  \arrow[d, "d_1"] \\
				K  \arrow[r] \arrow[ur, dashed]               &   X_{x/}
			\end{tikzcd}
		\end{equation*}
		in $ss\cS$ is equivalent to 
		\begin{equation}\label{eq:sx_is_left_eq_1}
			\begin{tikzcd}
				J \arrow[r] \arrow[d]      &   X_{\Delta^{1}_{s}/}  \arrow[d, "d_1"] \\
				K  \arrow[r]  \arrow[ur, dashed]              &   X_{\Delta^{0}_{s}/}
			\end{tikzcd}
		\end{equation}
		in $(ss\cS_{a})_{\ast}$, where $X_{\Delta^{1}_{s}/}$ and $X_{\Delta^{0}_{s}/}$ are pointed by $sx$ and $x$, respectively. Using naturality of $s$, we get that the composite map
		\begin{align*}
			X_{\Delta^{0}_{s}/} \xrightarrow{s^{2}} X_{\Delta^{2}_{s}/} \to X_{\Horn^{2}_{0,s}/}
		\end{align*}
		can be identified up to homotopy with 
		\begin{align*}
			s\timest s \colon X_{\Delta^{0}_{s}/} \to X_{\Delta^{1}_{s}/} \timest_{X_{\Delta^{0}_{s} / }} X_{\Delta^{1}_{s}/}\, .
		\end{align*}
		Therefore, $s^{2}$ factors through a map $X_{\Delta^{0}_{s} /} \to S^{*}X_{\Delta^{2}_{s}/}$. In particular, this means that for any vertex $x$, we can extend \cref{eq:sx_is_left_eq_1} as follows.
		\begin{equation*}
			\begin{tikzcd}
				\underline{\ast} \arrow[r, "s^{2}x"] \arrow[d] & S^{*} X_{\Delta^{2}_{s}/} \arrow[r] \arrow[d] & X_{\Delta^{2}_{s}/} \arrow[d] \\
				J \arrow[r] \arrow[d] \arrow[ur, dashed]     &   X_{\Delta^{1}_{s}/}    & X_{\Horn^{2}_{0,s}/} \arrow[d, "d_1"] \arrow[r, "d_{2}"] & X_{\Delta^{1}_{s}/} \arrow[d, "d_1"] \\
				K  \arrow[r] \arrow[uurr,dotted, bend right=15]       &   X_{\Delta^{0}_{s}/}  \arrow[r, "s"]            & X_{\Delta^{1}_{s}/} \arrow[r, "d_0"] & X_{\Delta^{0}_{s}/}
				\arrow[from=2-2, to=3-2,"d_1", crossing over]
				\arrow[from=2-2,to=2-3, "S"', crossing over]
			\end{tikzcd}
		\end{equation*}
		The dashed lift exists by \cref{lem:S*X_over_delta*2}. Then the dotted lift exists by \cref{cor:free_slice_fibration} because $d_1 \colon \Delta^{1}_{S} \to \Delta^{2}_{s}$ is right anodyne. The composites $d_2\circ S$ and $d_0 \circ s$ are homotopic to the identity, so we can push forward to obtain a solution of \cref{eq:sx_is_left_eq_1}.
	\end{proof}
	\begin{proof}[Proof of \cref{thm:outer_degenereacies}]
		Let $X$ be an inner Kan space which admits initial and terminal degeneracies $s_0$ and $s_{\omega}$. Then by \cref{lem:eq_fibration_condition}, and \cref{lem:sx_is_left_eq} and its dual, the edge $s_{0}x\sim s_{\omega} x$ is an equivalence in $X$. Moreover, the $2$-simplex $s^{2}x$ witnesses $sx$ as idempotent. 
	\end{proof}
	
	\printbibliography[heading=bibintoc, title=References]
\end{document}